\documentclass[a4paper,11pt,reqno]{amsart}
\usepackage{mathrsfs}
\usepackage[all]{xy}
\usepackage{amsmath,amssymb,amscd,bbm,amsthm,mathrsfs,dsfont,chemarrow}
 \numberwithin{equation}{section}
\newtheorem{thm}{Theorem}[section]
\newtheorem{lem}{Lemma}[section]
\newtheorem{cor}{Corollary}[section]
\newtheorem{pro}{Proposition}[section]

\theoremstyle{definition}

\theoremstyle{remark}
\newtheorem{rem}{Remark}[section]

 \numberwithin{equation}{section}
  
\begin{document}
 \title[The Bochner-Type Formula and The First Eigenvalue of the sub-Laplacian]{The Bochner-Type Formula and The First Eigenvalue of the sub-Laplacian on a Contact Riemannian Manifold}
 \author{Feifan Wu} \thanks{Department of Mathematics, Zhejiang University, Hangzhou 310027, P. R. China, Email: wesleywufeifan08@sina.com.} \author{Wei Wang} \thanks{Department of Mathematics, Zhejiang University, Hangzhou 310027, P. R. China, Email: wwang@zju.edu.cn.}
 \begin{abstract}
 Contact Riemannian manifolds, with not necessarily integrable complex structures, are the generalization of pseudohermitian manifolds in CR geometry. The Tanaka-Webster-Tanno connection on such a manifold  plays the role of Tanaka-Webster connection in the pseudohermitian case. We prove the contact Riemannian version of the pseudohermitian Bochner-type formula, and generalize the CR Lichnerowicz theorem about the sharp lower bound for the first nonzero eigenvalue of the sub-Laplacian to the contact Riemannnian case.
 \end{abstract}
\maketitle
\section{Introduction}
  Lichnerowicz \cite{Li1} obtained a sharp lower bound for the first eigenvalue of the Laplacian-Beltrami operator on a compact Riemannian manifold with a lower Ricci bound, and Obata \cite{Ob1} characterized the case of equality. On a pseudohermitian manifold, the sub-Laplacian is the counterpart of the Laplacian-Beltrami operator. The CR analogue of the Lichnerowicz theorem states that   for a $(2n+1)$-dimensional pseudohermitian manifold, $n\ge3$, satisfying \begin{equation}\label{lic}
  Ric(X,\overline{X})+\frac{n+1}{2}Tor(X,X){\ge}\kappa{h}(X,\overline{X}),
  \end{equation}
  the first nonzero eigenvalue of the sub-Laplacian is greater than or equal to $n\kappa/(n+1)$. This result was first proved by  Greenleaf \cite{G1}. But due to a mistake in calculation pointed out in \cite{CC1} and \cite{GL1}, the coefficient $\frac{n+1}{2}$ in \eqref{lic} was mistaken to be $\frac{n}{2}$. The corresponding results for $n=2$ and $n=1$ were obtained later in \cite{LL1} and \cite{Ch1}, respectively. The CR Obata-type theorem was conjectured in \cite{CC1}, which states that if $n\kappa/(n+1)$ is an eigenvalue of the sub-Laplacian on a pseudohermitian manifold, then it is the standard CR structure on the unit sphere in $\mathbb{C}^{n+1}$. This is proved under some additional conditions (cf. \cite{CC1}, \cite{CC2}, \cite{IV2} and references therein) and without conditions in \cite{LW1}. There is also a quaternionic contact version of Lichnerowicz theorem \cite{IPV1} (see e.g. \cite{BI1}, \cite{IV1} and \cite{Wa2} for the quaternionic contact manifolds). In this paper, we generalize the CR Lichnerowiecz theorem to the contact Riemannian case.

 A $(2n+1)$-dimensional manifold $M$ is called a {\it contact manifold} if it has a real 1-form $\theta$, called a {\it contact form}, such that $\theta{\wedge}d\theta^n\ne0$ everywhere on $M$. There exists a unique vector field $T$, {\it the Reeb vector field}, such that $\theta(T)=1$ and $T\lrcorner{d\theta}=0$. It is well known that given a contact manifold $(M,\theta)$, there are a Riemannian metric $h$ and a $(1,1)$-tensor field $J$ on $M$ such that
 \begin{equation}\label{eq2.1}
  \begin{aligned}
  &h(X,T)=\theta(X),\\
  &J^2=-Id+\theta\otimes{T},\\
  &d\theta(X,Y)=h(X,JY),
  \end{aligned}
 \end{equation}
 for any vector fields $X$ and $Y$ (cf. p. 278 in \cite{BD1}). We call $J$ an {\it almost complex structure}. Once $h$ is fixed, $J$ is uniquely determined. $(M,\theta,h,J)$ is called a {\it contact Riemannian manifold}.

 Let $TM$ be the tangent bundle and $\mathbb{C}TM$ be its complexification. Denote $HM:=Ker(\theta)$, the horizontal subbundle. $\mathbb{C}HM$ has a unique subbundle $T^{(1,0)}M$ such that $JX=iX$ for any $X\in{\Gamma(T^{(1,0)}M)}$. Here and in the following, $\Gamma(S)$ denotes the space of all sections of a vector bundle $S$. Set $T^{(0,1)}M=\overline{T^{(1,0)}M}$. For any $X\in{\Gamma(T^{(0,1)}M)}$, we have $JX=-iX$. $J$ is called {\it integrable} if $\big[\Gamma(T^{(1,0)}M),\Gamma(T^{(1,0)}M)\big]{\subset}\Gamma(T^{(1,0)}M).$
 In particular if $J$ is integrable, $J$ is called a {\it CR structure} and $(M,\theta,h,J)$ is called a {\it pseudohermitian manifold}. On a contact Riemannian manifold  there exists a distinguished connection  introduced in \cite{T1}, called the {\it Tanaka-Webster-Tanno connection} (or {\it TWT connection} briefly). In the pseudohermitian case, this connection is exactly the Tanaka-Webster connection. The Tanno tensor is defined as $Q=\nabla{J}$. $J$ is integrable if and only if the Tanno tensor $Q\equiv0$ (cf. Proposition 2.1 in \cite{T1}). We can also define the sub-Laplacian operator $\Delta_{b}$. Since there is no obstruction to the existence of the almost complex structure $J$, contact Riemannian structures exist naturally on any contact manifold  and analysis on it has potential applications to the geometry of contact manifolds (cf. \cite{BaD1}, \cite{Pe1} and \cite{Se1} and references therein).

 Since the Tanno tensor $Q$ is a (1,2)-tensor, $Q_X:=Q(X,\cdot)$ and $\nabla{Q}(X,X)$ are (1,1)-tensors. Define invariants of contact Riemannian structures:
 \begin{equation}\label{Q123}
 \begin{aligned}
 &Q_1(X,X)=-2Re\big(i{\cdot}{\rm trace}\big\{Y{\longrightarrow}\nabla_Y{Q}(X,X)\big\}\big),\\
 &Q_2(X,X)={\langle}Q_X,Q_{\bar{X}}{\rangle},\\
 &Q_3(X,X)={\rm trace}\big\{Y{\longrightarrow}Q_X{\circ}Q_{\bar{X}}(Y)\big\},
 \end{aligned}
 \end{equation}
 where $\langle\cdot,\cdot\rangle$ is the inner product on $(1,1)$-tensor induced by $h$, and
\begin{equation}\label{tor}
Tor(X,X)=2Re\big(i{h}(\tau_{\ast}X,X)\big),
\end{equation}
for any $X\in{T^{(1,0)}M}$, where $\tau_{\ast}$ is the {\it Webster torsion} defined as
$\tau_{\ast}(X)=\tau(T,X)$, $X\in{TM}.$

 Let $\nabla{u}$ be the gradient of $u$ with respect to the metric $h$, i.e., $h(\nabla{u},X)=Xu$ for any $X\in\Gamma(TM)$. Set $\nabla_{H}{u}=\pi_H\nabla{u}$, where $\pi_{H}$ is the orthogonal projection to $HM$. And $\partial_{b}u$ is the orthogonal projection of $\nabla_{H}{u}$ to $T^{(1,0)}M$. Our main result is as follows.
\begin{thm}\label{thm1.2}
On a $(2n+1)$-dimensional contact Riemannian manifold, $n\ge2$, we have the contact Riemannian Bochner-type formula
\begin{equation}\label{BTF}
\begin{aligned}
\Delta_{b}(\|\partial_{b}u\|^2)&=2\|\nabla^2u\|^2-4\nabla^2u(T,J\nabla_{H}{u})-2nTor(\partial_{b}u,\partial_{b}u)
+2Ric\bigg(\partial_{b}u,\overline{\partial_{b}}u\bigg)\\
&\ \ \ +h\bigg(\nabla_{H}{u},\nabla_{H}(\Delta_{b}u)\bigg)+Q_1(\partial_{b}u,\partial_{b}u)
-\frac{1}{2}Q_2(\partial_{b}u,\partial_{b}u),
\end{aligned}
\end{equation}
for any $u{\in}C_0^{\infty}(M)$.
\end{thm}
\begin{thm}\label{thm1.2b}
Suppose that on a compact $(2n+1)$-dimensional contact Riemannian manifold, $n\ge2$, there exists some positive constant $\kappa$ such that
\begin{equation}\label{eq1.2}
Ric(X,\overline{X})-(n+1)Tor(X,X)+\frac{1}{2}Q_1(X,X)-\frac{2n+7}{8(n-1)}Q_2(X,X)
+\frac{3}{2(n-1)}Q_3(X,X){\ge}\kappa{h}(X,\overline{X}),
\end{equation}
for any $X\in{T^{(1,0)}M}$, where $Ric$ is the Ricci tensor. Then the first nonzero eigenvalue $\lambda_1$ of $\Delta_{b}$ satisfies
$$\lambda_1\ge\frac{{\kappa}n}{n+1}.$$
\end{thm}
Note that the coefficients of $\ Tor$ in \eqref{BTF} and \eqref{eq1.2} are different from that \eqref{lic} by a factor $-2$ (cf. \cite{DT1} and \cite{LW1}) . This is because that in our definition \eqref{eq2.1}, $d\theta(X,Y)=h(X,JY)$, while in psuedohermitian case, people usually use $d\theta(X,Y)=2h(JX,Y)= -2 h(X,JY)$. When $Q\equiv 0$, Theorem \ref{thm1.2} and \ref{thm1.2b} coincide with the CR Bochner-type formula and CR Lichnerowiecz theorem, respectively (see e.g. \cite{CC1}, \cite{DT1}, \cite{GL1}, \cite{G1}, \cite{LW1}). It is quite interesting to characterize the equality case of \eqref{eq1.2}.

In Section 2, we introduce some basic preliminaries, including the TWT connection, the torsion tensor, the curvature tensor and the Tanno tensor. If we choose an orthonormal $T^{(1,0)}M$ frame, there are some simpler relations for the connection coefficients, the Tanno tensor and the structure equations, which will make our calculation easier.

When given an orthonormal $T^{(1,0)}M$ frame, we have $\Gamma_{\alpha\beta}^{\bar{\gamma}}=-\frac{i}{2}Q_{\beta\alpha}^{\bar{\gamma}}$, which vanish in the pseudohermitian case. But in the general case, it may not always vanish. Therefore there exists extra terms involving such connection coefficients in our formulae, e.g. the Bochner-type formula and various integral identities, which will make our calculation more complicated than the pseudohermitian case. The main difficulties of generalizing results to the contact Riemannian case come from handling such extra terms.

In section 3, we introduce the second- and third-order covariant derivatives and their commutation formulae with respect to an orthonormal $T^{(1,0)}M$ frame. In Section 4, we prove the Bochner-type formula on a contact Riemannian manifold. This formula differs from the pseudohermitian case by terms involving the Tanno tensor. And it coincides with the CR Bochner-type formula (cf. e.g. Proposition 9.5 in \cite{DT1} or Theorem 6 in \cite{LW1}) when the almost complex structure $J$ is integrable. Similarly to pseudohermitian case, the term $\nabla^2u(T,J\nabla_{H}{u})$ in the Bochner-type formula can be controlled by using two integral identities. But here, in one identity, we have to use another identity to handle extra terms depending on the Tanno tensor $Q$. It is done in Section 5. In Section 6, with the preparation above, we prove the main Theorem \ref{thm1.2b}.

\section{Connection coefficients, torsions and curvatures on contact Riemannian manifolds}
\subsection{TWT connection, the Tanno tensor and the orthonormal $T^{(1,0)}M$ frame}
  \begin{pro}(cf. (7)-(10) in \cite{BD1})\label{prop2.0a}
  On a contact Riemannian manifold $(M,\theta,h,J)$, there exists a unique linear connection such that
  \begin{equation}\label{TWT}
  \begin{split}
  &\nabla{\theta}=0,\quad\nabla{T}=0,\\
  &\nabla{h}=0,\\
  &\tau(X,Y)=2d\theta(X,Y)T,{\quad}X,Y\in{\Gamma(HM)},\\
  &\tau(T,JZ)=-J{\tau(T,Z)},{\quad}Z\in{{\Gamma(TM)}},
  \end{split}
  \end{equation}
  where $\tau$ is the torsion of $\nabla$.
  \end{pro}
  $\nabla$ is called the {\it TWT connection}.
  The {\it Tanno tensor} $Q$ (cf. (10) in \cite{DT1}) is defined as
  \begin{equation}\label{eq2.8}
  Q(X,Y):=(\nabla_YJ)X,{\quad}{\text for\ }X,Y\in{\Gamma(TM)}.
  \end{equation}

  We extend $h$, $J$ and $\nabla$ to the complexified tangent bundle by $\mathbb{C}$-linear extension:
  \begin{equation}
  \begin{aligned}
  h(X_1+iY_1,X_2+iY_2)&:=h(X_1,X_2)-h(Y_1,Y_2)+i\big(h(X_1,Y_2)+h(X_2,Y_1)\big),\\
  \notag J(X_1+iY_1)&:=JX_1+iJY_1,\\
  \nabla_{(X_1+iY_1)}(X_2+iY_2)&:=\nabla_{X_1}X_2-\nabla_{Y_1}Y_2+i\big(\nabla_{X_1}Y_2+\nabla_{Y_1}X_2\big).
  \end{aligned}
  \end{equation}
  for any $Z_j=X_j+iY_j\in{\mathbb{C}TM}$, $j=1,2$.
\begin{pro}\label{pro2.0}
Let $W_0:=T$, the Reeb vector. We can choose a local $T^{(1,0)}M$-frame $\{W_j\}=\{W_a,W_0\}=\{W_{\alpha},W_{\bar{\alpha}},T\}$ with $W_{\alpha}\in{T^{(1,0)}M}$, $W_{\bar{\alpha}}=\overline{W_{\alpha}}{\in}T^{(0,1)}M$ on  a neighborhood $U$ such that
$$h_{\alpha\bar{\beta}}=\delta_{\alpha\bar{\beta}};{\qquad}h_{\bar{\alpha}\beta}=\delta_{\bar{\alpha}\beta}=\delta_{\alpha\bar{\beta}};
{\qquad}h_{\alpha\beta}=0.$$
We call this frame an {\it orthonormal $T^{(1,0)}M$-frame}.
\end{pro}
\begin{proof}
Note that \eqref{eq2.1} leads to
\begin{equation}\label{eq2.7a}
\begin{aligned}
 &JT=0,{\quad}\theta(JX)=0,\\
 &h(X,Y)=h(JX,JY)+\theta(X)\theta(Y),{\quad}d\theta(X,JY)=-d\theta(JX,Y),
 \end{aligned}
 \end{equation}
for any $X,Y\in{TM}$ (cf. p. 351 in \cite{T1}). Choose a vector field $X_1$ in $\Gamma(HM)$ such that $h(X_1,X_1)=\frac{1}{2}$ and let $X_{n+1}:=JX_1$. Then
  $h(X_1,X_{n+1})=h(X_1,JX_1)=d\theta(X_1,X_1)=0$, i.e. $X_{n+1}$ is automatically orthogonal to $X_1$,
  and by third identity in \eqref{eq2.7a}, we get $h(X_{n+1},X_{n+1})=h(JX_1,JX_1)=h(X_1,X_1)=\frac{1}{2}$. We choose $X_2$ orthogonal to $span\{X_1,JX_1\}$, and define $X_{n+2}:=JX_2$. Repeating the procedure, we find a local orthogonal basis $X_1,\cdots,X_{2n}$ with $h(X_a,X_b)=\frac{1}{2}\delta_{ab}$ and $JX_{\alpha}=X_{\alpha+n}$.
  Now define
  \begin{equation}\label{eq2.2a}
  W_{\alpha}:=X_{\alpha}-iX_{\alpha+n},{\qquad}W_{\bar{\alpha}}:=\overline{W_{\alpha}}.
  \end{equation}
  It is direct to see that $JW_{\alpha}=iW_{\alpha}$ and $JW_{\bar{\alpha}}=-iW_{\bar{\alpha}}$. Namely, $W_{\alpha}\in{T^{(1,0)}M}$ and $W_{\bar{\alpha}}\in{T^{(0,1)}M}$. Then by Remark \ref{rem2.0} for the complex extension we get $h(W_{\alpha},W_{\beta})=h(X_{\alpha}-iX_{\alpha+n},X_{\beta}-iX_{\beta+n})=0$, $h_{\bar{\alpha}\bar{\beta}}=0$
  and $\notag h(W_{\alpha},W_{\bar{\beta}})=\delta_{\alpha\bar{\beta}}$, $h_{\bar{\alpha}\beta}=\delta_{\bar{\alpha}\beta}=\delta_{\alpha\bar{\beta}}$.
\end{proof}
  \begin{rem}\label{rem2.0}
  (1) For the multi-index, we adopt the following index conventions in this paper.
  \begin{equation*}
  \begin{aligned}
  &\alpha,\beta,\gamma,\rho,\lambda,\mu,\cdots\in\{1,\cdots,n\},{\qquad}a,b,c,d,e,\cdots\in\{1,2,\cdots,2n\},\\
  &j,k,l,r,s,\cdots\in\{0,1,\cdots,2n\},{\qquad}\bar{\alpha}=\alpha+n.
  \end{aligned}
  \end{equation*}

  (2) In this paper, the Einstein summation convention will be used. Moreover, if indices $\alpha$ and $\bar{{\alpha}}$ both appear in low (or upper) indices, then the index $\alpha$ will be taken summation, e.g. $h_{\alpha\bar{\alpha}}=\sum\limits_{\alpha}h_{\alpha\bar{\alpha}}.$
  \end{rem}
From now on, we choose a local orthonormal $T^{(1,0)}M$-frame $\{W_j\}$. In particular, by \eqref{eq2.1}, $h(T,W_a)=h(W_a,T)=\theta(W_a)=0$ and $h(T,T)=\theta(T)=1$. We denote $h_{ab}=h(W_a,W_b)$ and use $h_{ab}$ and its inverse matrix to lower and raise indices.  Let $\{\theta^{\beta},\theta^{\bar{\beta}},\theta\}$ denote the dual coframe to $\{W_{\alpha},W_{\bar{\alpha}},T\}$, i.e.,
  $\theta^{\beta}(W_{\alpha})=\delta_{\alpha}^{\beta}$, $\theta^{\beta}(W_{\bar{\alpha}})=\theta^{\beta}(T)=0$, $\theta^{\bar{\beta}}:=\overline{\theta^{\beta}}$, and
  $\theta(W_{\alpha})=\theta(W_{\bar{\alpha}})=0$, $\theta(T)=1$. Set $\theta^0:=\theta$. The connection 1-form with respect to $\{W_j\}$ is given by $\nabla{W_j}=\omega_{j}^{k}\otimes{W_k}$, and set $\omega_{j}^{k}:=\Gamma_{ij}^k\theta^i$, i.e. $\nabla_{W_i}W_j=\Gamma_{ij}^kW_k$. By \eqref{TWT}, we get $\theta(\nabla{X})=h(\nabla{X},T)=d(h(X,T))-h(X,\nabla{T})=0$ for any $X\in{TM}$, namely $\Gamma_{ij}^0=0$. And $\nabla{T}=0$ implies $\Gamma_{i0}^k=0$.

By the dual argument, we have
$$\nabla_{W_i}\theta^k=-\Gamma_{ij}^k\theta^j.$$
And for any $(r,s)$-tensor $\varphi$ with components $\varphi_{\phantom{k_1{\cdots}k_r}j_1{\cdots}j_s}^{k_1{\cdots}k_r}=\varphi(\theta^{k_1},\cdots,\theta^{k_r},W_{j_1},\cdots,W_{j_s})$, covariant derivatives of $\varphi$ are given by
\begin{align}\label{tc} \varphi_{\phantom{k_1{\cdots}k_r}j_1{\cdots}j_s,i}^{k_1{\cdots}k_r}&=W_i(\varphi_{\phantom{k_1{\cdots}k_r}j_1{\cdots}j_s}^{k_1{\cdots}k_r})
+\sum\limits_{t=1}^r\Gamma_{il}^{k_t}\varphi_{\phantom{k_1{\cdots}l{\cdots}k_r}j_1{\cdots}j_s}^{k_1{\cdots}l{\cdots}k_r}
-\sum\limits_{t=1}^s\Gamma_{ij_t}^{l}\varphi_{\phantom{k_1{\cdots}k_r}j_1{\cdots}l{\cdots}j_s}^{k_1{\cdots}k_r}.
\end{align}
Denote the components of the almost complex structure $J$ by $J_{\ j}^l$. We write $Q(W_j,W_k)=(\nabla_{W_k}J)W_j=Q_{jk}^lW_l$. Equivalently, $Q_{jk}^l:=J_{\ j,k}^l$. Applying \eqref{tc}, we have
\begin{equation}\label{CQ}
\begin{aligned}
&Q_{jk}^l=W_kJ_{\ j}^l+\Gamma_{ks}^lJ_{\ j}^s-\Gamma_{kj}^sJ_{\ s}^l,\\
&Q_{jk,s}^l=W_sQ_{jk}^l-\Gamma_{sj}^rQ_{rk}^l-\Gamma_{sk}^rQ_{jr}^l+\Gamma_{sr}^lQ_{jk}^r.
\end{aligned}
\end{equation}
\begin{pro}\label{van} (cf. (16)-(18) in \cite{BD1}) With respect to a local $T^{(1,0)}$-frame $\{W_j\}$, the components of tensor $Q$ has the following property:
\begin{equation}\label{QBD}
\begin{aligned}
&Q_{\beta\alpha}^{\gamma}=0,{\quad}\Gamma_{\alpha\bar{\beta}}^{\gamma}=0,{\quad}Q_{\bar{\beta}\alpha}^{\gamma}=0,{\quad}
Q_{\bar{\beta}\alpha}^{\bar{\gamma}}=0,{\quad}\Gamma_{\alpha\beta}^{\bar{\gamma}}=-\frac{i}{2}Q_{\beta\alpha}^{\bar{\gamma}},\\
&\Gamma_{\alpha\bar{\beta}}^0=0,{\quad}\Gamma_{\alpha\beta}^0=0,{\quad}\Gamma_{\alpha0}^{\gamma}=0,
{\quad}\Gamma_{\alpha0}^{\bar{\gamma}}=0{\quad}\Gamma_{0\beta}^{\bar{\gamma}}=0,{\quad}\Gamma_{0\beta}^0=0,\\
&Q_{0j}^k=0,{\quad}Q_{i0}^k=0,{\quad}Q_{ij}^0=0.
\end{aligned}
\end{equation}
\end{pro}
In particular, only the components $Q_{\beta\alpha}^{\bar{\gamma}}$ of tensor $Q$ are non-vanishing. In \eqref{QBD}, $Q_{0j}^k=0$ follows from $Q_{0j}^kW_k=Q(T,W_j)=(\nabla_{W_j}J)T=\nabla_{W_j}(JT)-J\nabla_{W_j}T\equiv0$ by \eqref{TWT} and \eqref{eq2.7a} for any $W_j$. And $Q_{i0}^k=0$ follows from setting $X=W_i$, $Y=T$, $Z=W_l$ in identity (cf. (15) in \cite{BD1}):
\begin{equation}\label{eq2.13}
2h(Q(X,Y),Z)=h(N^{(1)}(X,Z)-\theta(X)N^{(1)}(T,Z)+\theta(Z)N^{(1)}(X,T),JY),
\end{equation}
for any $X,Y,Z\in{TM}$, where
$$N^{(1)}=[J,J]+2(d\theta)\otimes{T},{\quad}[J,J](X,Y)=J^2[X,Y]+[JX,JY]-J[JX,Y]-J[X,JY],$$
to get $Q_{i0}^kh_{kl}\equiv0$ by $JT\equiv0$ in \eqref{eq2.7a}. For $Q_{ij}^0=0$, since we already have $Q_{i0}^0=Q_{0j}^0=0$, it remains to prove $Q_{ab}^0=0$. This follows from $Q_{ab}^0=\theta\bigg(Q(W_a,W_b)\bigg)=\theta\bigg((\nabla_{W_b}J)W_a\bigg)=h(T,(\nabla_{W_b}J)W_a)
=h(T,\nabla_{W_b}{(JW_a)})-h(T,J\nabla_{W_b}W_a)
=\theta\bigg(\nabla_{W_b}{(JW_a)}\bigg)-\theta(J\nabla_{W_b}W_a)=0$ by $W_a$, $W_b$ being horizontal.
\begin{rem}
In pseudohermitian case, $\Gamma_{\alpha\beta}^{\bar{\gamma}}=0$ by the Tanaka-Webster connection preserving $T^{(1,0)}M$. But in general case, $\Gamma_{\alpha\beta}^{\bar{\gamma}}=-\frac{i}{2}Q_{\beta\alpha}^{\bar{\gamma}}$ may not vanish.
\end{rem}
Recall that only the components $Q_{\beta\alpha}^{\bar{\gamma}}$ of tensor $Q$ are non-vanishing. So by definition \eqref{Q123}, with respect to a local orthonormal $T^{(1,0)}$-frame $\{W_j\}$, we have
\begin{equation}\label{GQ}
\begin{aligned}
Q_1(X,X)&=-2Re\bigg(i{\cdot}{\rm trace}\big\{W_j{\longrightarrow}\nabla_{W_j}{Q}(X,X)\big\}\bigg)
=-2Re\bigg(iQ_{\alpha\beta,j}^{j}X^{\alpha}X^{\beta}\bigg)\\
&=-2Re\bigg(iQ_{\alpha\beta,\bar{\gamma}}^{\bar{\gamma}}X^{\alpha}X^{\beta}\bigg)
=i\bigg(Q_{\bar{\alpha}\bar{\beta},\gamma}^{\gamma}X^{\bar{\alpha}}X^{\bar{\beta}}
-Q_{\alpha\beta,\bar{\gamma}}^{\bar{\gamma}}X^{\alpha}X^{\beta}\bigg),\\
Q_2(X,X)&=h_{ij}h^{kl}Q_{\alpha{k}}^{i}Q_{\bar{\beta}l}^{j}X^{\alpha}X^{\bar{\beta}}
=h_{\lambda\bar{\rho}}h^{\gamma\bar{\mu}}Q_{\alpha\gamma}^{\bar{\rho}}Q_{\bar{\beta}\bar{\mu}}^{\lambda}X^{\alpha}X^{\bar{\beta}}
=Q_{\alpha\gamma}^{\bar{\rho}}Q_{\bar{\beta}\bar{\gamma}}^{\rho}X^{\alpha}X^{\bar{\beta}},\\
Q_3(X,X)&={\rm trace}\big\{W_j{\longrightarrow}Q_{W_{\alpha}}{\circ}Q_{W_{\bar{\beta}}}(W_j)X^{\alpha}X^{\bar{\beta}}\big\}\\
&={\rm trace}
\big\{W_j{\longrightarrow}Q_{\bar{\beta}j}^{i}Q_{W_{\alpha}}(W_i)X^{\alpha}X^{\bar{\beta}}\big\}
=Q_{\alpha{i}}^{j}Q_{\bar{\beta}{j}}^{i}X^{\alpha}X^{\bar{\beta}}
=Q_{\alpha\rho}^{\bar{\gamma}}Q_{\bar{\beta}\bar{\gamma}}^{\rho}X^{\alpha}X^{\bar{\beta}},
\end{aligned}
\end{equation}
for any $X=X^{\alpha}W_{\alpha}\in{T^{(1,0)}M}$. According to \eqref{CQ} and $\Gamma_{\alpha\bar{\beta}}^{\gamma}=0$ in \eqref{QBD}, we get the components
\begin{equation}\label{eq2.86} Q_{\alpha\beta,\bar{\rho}}^{\bar{\gamma}}=W_{\bar{\rho}}Q_{\alpha\beta}^{\bar{\gamma}}-\Gamma_{\bar{\rho}\alpha}^eQ_{e\beta}^{\bar{\gamma}}
-\Gamma_{\bar{\rho}\beta}^eQ_{\alpha{e}}^{\bar{\gamma}}+\Gamma_{\bar{\rho}e}^{\bar{\gamma}}Q_{\alpha\gamma}^e
=W_{\bar{\rho}}Q_{\alpha\beta}^{\bar{\gamma}}-\Gamma_{\bar{\rho}\alpha}^{\mu}Q_{\mu\beta}^{\bar{\gamma}}
-\Gamma_{\bar{\rho}\beta}^{\mu}Q_{\alpha\mu}^{\bar{\gamma}}+\Gamma_{\bar{\rho}\bar{\mu}}^{\bar{\gamma}}Q_{\alpha\beta}^{\bar{\mu}},
\end{equation}
of tensor $\nabla{Q}$.
\begin{pro}\label{pro2.1}
With respect to a local orthonormal $T^{(1,0)}M$-frame $\{W_j\}$, we have
\begin{equation}\label{eq2.9}
\begin{aligned}
&\Gamma_{\alpha\beta}^{\gamma}=-\Gamma_{\alpha\bar{\gamma}}^{\bar{\beta}},{\qquad}\Gamma_{\alpha\beta}^{\bar{\gamma}}=-\Gamma_{\alpha\gamma}^{\bar{\beta}},\\
&Q_{\beta\alpha}^{\bar{\gamma}}=-Q_{\gamma\alpha}^{\bar{\beta}},{\qquad}Q_{\alpha\beta,\bar{\rho}}^{\bar{\gamma}}=-Q_{\gamma\beta,\bar{\rho}}^{\bar{\alpha}},\\
&Q_{\alpha\beta}^{\bar{\gamma}}=Q_{\alpha\gamma}^{\bar{\beta}}-Q_{\gamma\alpha}^{\bar{\beta}},
\end{aligned}
\end{equation}
and their conjugation.
\end{pro}
\begin{proof}
By \eqref{QBD} and Proposition \ref{pro2.0}, we have
$$\Gamma_{\alpha\beta}^{\gamma}=\Gamma_{\alpha\beta}^{\rho}{\delta}_{\rho\bar{\gamma}}=h(\nabla_{W_{\alpha}}W_{\beta},W_{\bar{\gamma}})
=W_{\alpha}(h_{\beta{\bar{\gamma}}})-h(W_{\beta},\nabla_{W_{\alpha}}W_{\bar{\gamma}})=-h_{\beta\bar{\mu}}\Gamma_{\alpha\bar{\gamma}}^{\bar{\mu}}=-\Gamma_{\alpha\bar{\gamma}}^{\bar{\beta}}.$$
$\Gamma_{\alpha\beta}^{\bar{\gamma}}=-\Gamma_{\alpha\gamma}^{\bar{\beta}}$ follows similarly. Then we get
$Q_{\beta\alpha}^{\bar{\gamma}}=2i\Gamma_{\alpha\beta}^{\bar{\gamma}}=-2i\Gamma_{\alpha\gamma}^{\bar{\beta}}=-Q_{\gamma\alpha}^{\bar{\beta}}$ by \eqref{QBD}. The fourth identity in \eqref{eq2.9} follows from this identity.

For the last identity in \eqref{eq2.9}, setting $X=W_{\alpha}$, $Y=W_{\beta}$, $Z=W_{\gamma}$ in \eqref{eq2.13}, we get
\begin{align}
\notag 2h(Q(W_{\alpha},W_{\beta}),W_{\gamma})&=h([J,J](W_{\alpha},W_{\gamma}),JW_{\beta})\\
\notag &=ih(J^2[W_{\alpha},W_{\gamma}]+[JW_{\alpha},JW_{\gamma}]-J[JW_{\alpha},W_{\gamma}]-J[W_{\alpha},JW_{\gamma}],W_{\beta})\\
\notag &=-2ih([W_{\alpha},W_{\gamma}],W_{\beta})+2h(J[W_{\alpha},W_{\gamma}],W_{\beta})=-4ih([W_{\alpha},W_{\gamma}],W_{\beta}).
\end{align}
By the definition of the torsion tensor and \eqref{QBD}, the $T^{(0,1)}M$-components of $[W_{\alpha},W_{\gamma}]$ is given by
\begin{align}
\notag [W_{\alpha},W_{\gamma}]&=\nabla_{W_{\alpha}}W_{\gamma}-\nabla_{W_{\gamma}}W_{\alpha}-\tau(W_{\alpha},W_{\gamma})=\Gamma_{\alpha\gamma}^{\bar{\rho}}W_{\bar{\rho}}-\Gamma_{\gamma\alpha}^{\bar{\rho}}W_{\bar{\rho}}\\
&=-\frac{i}{2}Q_{\gamma\alpha}^{\bar{\rho}}W_{\bar{\rho}}+\frac{i}{2}Q_{\alpha\gamma}^{\bar{\rho}}W_{\bar{\rho}}\quad{mod}{\quad}W_{\rho},{\quad}T.
\end{align}
Therefore, we get
\begin{align}
\notag 2Q_{\alpha\beta}^{\bar{\mu}}h_{\gamma\bar{\mu}}=2h(Q(W_{\alpha},W_{\beta}),W_{\gamma})&
=-4ih([W_{\alpha},W_{\gamma}],W_{\beta})=-2Q_{\gamma\alpha}^{\bar{\rho}}h_{\beta\bar{\rho}}+2Q_{\alpha\gamma}^{\bar{\rho}}h_{\beta\bar{\rho}}.
\end{align}
The last identity of \eqref{eq2.9} holds.
\end{proof}
\subsection{The Webster torsion, the curvature tensor and the structure equations}
\begin{lem}\label{lem2.1}(cf. Lemma 1 in \cite{BD1})
The Webster torsion has following properties:
\begin{equation}
\notag \tau_{\ast}(T)=0;{\quad}\tau_{\ast}T_{(1,0)}M{\subseteq}T_{(0,1)}M;{\quad}\tau_{\ast}T_{(0,1)}M{\subseteq}T_{(1,0)}M.
\end{equation}
\end{lem}
By Lemma \ref{lem2.1}, we can write $\tau_{\ast}(W_{\alpha})=A_{\alpha}^{\bar{\beta}}W_{\bar{\beta}}$. Set $\tau^{\alpha}:=A_{\bar{\beta}}^{\alpha}\theta^{\bar{\beta}}$. So by \eqref{tor}, with respect to $\{W_j\}$, we have
\begin{equation}\label{tor1}
Tor(X,X)=2Re\bigg(iA_{\alpha\beta}X^{\alpha}X^{\beta}\bigg)
=iA_{\alpha\beta}X^{\alpha}X^{\beta}-iA_{\bar{\alpha}\bar{\beta}}X^{\bar{\alpha}}X^{\bar{\beta}},
\end{equation}
for any $X=X^{\alpha}W_{\alpha}\in{T^{(1,0)}M}$.

The components $R_{j\ kl}^{\ s}$ of the curvature tensor $R(X,Y)=\nabla_X\nabla_Y-\nabla_Y\nabla_X-\nabla_{[X,Y]}$ is given by $R(W_k,W_l)W_j=R_{j\ kl}^{\ s}W_s$. The Ricci tensor is given by
$$Ric(Y,Z)={\rm trace}\{X\longrightarrow{R(X,Z)Y}\},$$
for any $X,Y,Z\in{TM}$ (cf. p. 299 in \cite{BD1}). And the scalar curvature is $R=trace(Ric)$. With respect to a $T^{(1,0)}M$-frame, $R_{\alpha\bar{\beta}}=R_{\alpha\ \gamma\bar{\beta}}^{\ \gamma}$ (cf. (53) in \cite{BD1}). The scalar curvature is $R=h^{\alpha\bar{\beta}}R_{\alpha\bar{\beta}}$.

\begin{pro} (cf. (13), (14) and (39) in \cite{BD1}) With respect to a local orthonormal $T^{(1,0)}M$-frame $\{W_j\}$, we have the following structure equations:
\begin{equation}\label{SE}
\begin{split}
&d\theta=-2ih_{\alpha\bar{\beta}}\theta^{\alpha}\wedge\theta^{\bar{\beta}}=-2i\theta^{\alpha}\wedge\theta^{\bar{\alpha}},\\
&d\theta^{\alpha}=\theta^b\wedge{\omega_b^{\alpha}}+\theta\wedge{\tau^{\alpha}}=\theta^b\wedge\omega_b^{\alpha}+A_{\bar{\beta}}^{\alpha}\theta\wedge\theta^{\bar{\beta}},\\
&d\omega_a^b-\omega_a^c\wedge\omega_c^b=R_{a\ \lambda\bar{\mu}}^{\ b}\theta^{\lambda}\wedge\theta^{\bar{\mu}}+\frac{1}{2}R_{a\ \lambda\mu}^{\ b}\theta^{\lambda}\wedge\theta^{\mu}+\frac{1}{2}R_{a\ \bar{\lambda}\bar{\mu}}^{\ b}\theta^{\bar{\lambda}}\wedge\theta^{{\bar\mu}}+R_{a\ 0\bar{\mu}}^{\ b}\theta\wedge\theta^{\bar{\mu}}-R_{a\ \lambda0}^{\ b}\theta\wedge\theta^{\lambda}\\
&R(X,Y)W_a=2\big(d\omega_a^b-\omega_a^c\wedge\omega_c^b\big)(X,Y)W_b.
\end{split}
\end{equation}
\end{pro}
Here following \cite{BD1} we use the following definition for exterior product and exterior derivatives
\begin{equation}\label{eq2.15a}
\begin{aligned}
&\phi\wedge\psi(X,Y)=\frac{1}{2}\bigg(\phi(X)\psi(Y)-\psi(X)\phi(Y)\bigg),\\
&2(d\phi)(X,Y)=X(\phi(Y))-Y(\phi(X))-\phi([X,Y])=(\nabla_X\phi)Y-(\nabla_Y\phi)X+\phi(\tau(X,Y)),
\end{aligned}
\end{equation}
for any $1$-form $\phi$ and $\psi$. The second identity in \eqref{SE} follows from the orthonormality of $\{W_a\}$.
\begin{cor}
With respect to a local orthonormal $T^{(1,0)}M$-frame $\{W_j\}$, set $J_{ab}=h_{ac}J_{\ b}^c$. We have
\begin{equation}\label{eq2.8a}
R_{a\ cd}^{\ b}=W_c\Gamma_{da}^b-W_d\Gamma_{ca}^b-\Gamma_{cd}^e\Gamma_{ea}^b+\Gamma_{dc}^e\Gamma_{ea}^b-\Gamma_{ca}^e\Gamma_{de}^b+\Gamma_{da}^e\Gamma_{ce}^b+2\Gamma_{0a}^bJ_{cd}.
\end{equation}
\end{cor}
\begin{proof}
Note that $h(W_a,JW_b)=h(W_a,J_{\ b}^cW_c)=h_{ac}J_{\ b}^c=J_{ab}$. By \eqref{TWT} and the last identity in \eqref{SE}, we have
\begin{align}
\notag R_{a\ cd}^{\ b}&=2(d\omega_a^b)(W_c,W_d)-2\omega_a^e\wedge\omega_e^b(W_c,W_d)\\
\notag &=(\nabla_{W_c}\omega_a^b)(W_d)-(\nabla_{W_d}\omega_a^b)(W_c)+\omega_a^b(\tau(W_c,W_d))-\Gamma_{ca}^e\Gamma_{de}^b+\Gamma_{da}^e\Gamma_{ce}^b\\
\notag &=W_c\Gamma_{da}^b-W_d\Gamma_{ca}^b-\omega_a^b(\nabla_{W_c}W_d)+\omega_a^b(\nabla_{W_d}W_c)+\omega_a^b(2h(W_c,JW_d)T)-\Gamma_{ca}^e\Gamma_{de}^b+\Gamma_{da}^e\Gamma_{ce}^b\\
\notag &=W_c\Gamma_{da}^b-W_d\Gamma_{ca}^b-\Gamma_{cd}^e\Gamma_{ea}^b+\Gamma_{dc}^e\Gamma_{ea}^b-\Gamma_{ca}^e\Gamma_{de}^b+\Gamma_{da}^e\Gamma_{ce}^b+2\Gamma_{0a}^bJ_{cd}.
\end{align}
\end{proof}
\begin{pro}\label{pro2.5}
For the components of the curvature tensor, we have the following commutation relations:
\begin{equation}\label{COMR}
\begin{aligned}
R_{\alpha\bar{\beta}\gamma\bar{\mu}}=-R_{\alpha\bar{\beta}\bar{\mu}\gamma},{\quad}
R_{\alpha\bar{\beta}\gamma\bar{\mu}}=-R_{\bar{\beta}\alpha\gamma\bar{\mu}},{\quad}
R_{\alpha\bar{\beta}\gamma\bar{\mu}}=R_{\gamma\bar{\beta}\alpha\bar{\mu}},
\end{aligned}
\end{equation}
and their conjugation with respect to a local orthonormal $T^{(1,0)}M$-frame $\{W_j\}$.
\end{pro}
\begin{proof}
The first identity in \eqref{COMR} follows directly by the definition of the curvature tensor. For the last one in \eqref{COMR}, we refer to Corollary 1 in \cite{BD1}. For the second identity in \eqref{COMR}, note that $\nabla{h}=0$ implies $Xh(Y,Z)=h(\nabla_XY,Z)+h(Y,\nabla_XZ)$ and $h_{ab}$ are constants. Then
\begin{align}
\notag R_{\alpha\bar{\beta}\gamma\bar{\mu}}&=h\big(\nabla_{W_{\gamma}}\nabla_{W_{\bar{\mu}}}W_{\alpha}-\nabla_{W_{\bar{\mu}}}\nabla_{W_{\gamma}}W_{\alpha}-\nabla_{[W_{\gamma},W_{\bar{\mu}}]}W_{\alpha},W_{\bar{\beta}}\big)\\
\notag &=W_{\gamma}\bigg(h\big(\nabla_{W_{\bar{\mu}}}W_{\alpha},W_{\bar{\beta}}\big)\bigg)-h\big(\nabla_{W_{\bar{\mu}}}W_{\alpha},\nabla_{W_{\gamma}}W_{\bar{\beta}}\big)
-W_{\bar{\mu}}\bigg(h\big(\nabla_{W_{\gamma}}W_{\alpha},W_{\bar{\beta}}\big)\bigg)\\
\notag &\ \ \ +h\big(\nabla_{W_{\gamma}}W_{\alpha},\nabla_{W_{\bar{\mu}}}W_{\bar{\beta}}\big)-[W_{\gamma},W_{\bar{\mu}}]\big(h(W_{\alpha},W_{\bar{\beta}})\big)+h(W_{\alpha},\nabla_{[W_{\gamma},W_{\bar{\mu}}]}W_{\bar{\beta}})\\
\notag &=W_{\gamma}W_{\bar{\mu}}\big(h_{\alpha\bar{\beta}}\big)-W_{\gamma}\big(h(W_{\alpha},\nabla_{W_{\bar{\mu}}}W_{\bar{\beta}})\big)
-W_{\bar{\mu}}\big(h(W_{\alpha},\nabla_{W_{\gamma}}W_{\bar{\beta}})\big)\\
\notag &\ \ \ +h(W_{\alpha},\nabla_{W_{\bar{\mu}}}\nabla_{W_{\gamma}}W_{\bar{\beta}})-W_{\bar{\mu}}W_{\gamma}\big(h_{\alpha\bar{\beta}}\big)
+W_{\bar{\mu}}\big(h(W_{\alpha},\nabla_{W_{\gamma}}W_{\bar{\beta}})\big)\\
\notag &\ \ \ +W_{\gamma}\big(h(W_{\alpha},\nabla_{W_{\bar{\mu}}}W_{\bar{\beta}})\big)-h(W_{\alpha},\nabla_{W_{\gamma}}\nabla_{W_{\bar{\mu}}}W_{\bar{\beta}})
-[W_{\gamma},W_{\bar{\mu}}]\big(h_{\alpha\bar{\beta}}\big)+h(W_{\alpha},\nabla_{[W_{\gamma},W_{\bar{\mu}}]}W_{\bar{\beta}})\\
\notag &=h(W_{\alpha},\nabla_{W_{\bar{\mu}}}\nabla_{W_{\gamma}}W_{\bar{\beta}})-h(W_{\alpha},\nabla_{W_{\gamma}}\nabla_{W_{\bar{\mu}}}W_{\bar{\beta}})+h(W_{\alpha},\nabla_{[W_{\gamma},W_{\bar{\mu}}]}W_{\bar{\beta}})
=-R_{\bar{\beta}\alpha\gamma\bar{\mu}}.
\end{align}
\end{proof}
\begin{rem}\label{rem2.2}
By Remark \ref{rem2.0} for the complex extension, it's easy to see that  under the complex conjugation, the Riemannian metric $h$, the almost complex structure $J$, the TWT connection $\nabla$, the torsion tensor $A$, the curvature tensor $R$ and the Tanno tensor $Q$ are preserved, i.e.,
  $$\overline{h(Z_1,Z_2)}=h(\overline{Z_1},\overline{Z_2}),{\quad}J\overline{Z_1}=\overline{JZ_1},{\quad}\overline{\nabla_{Z_1}Z_2}=\nabla_{\overline{Z_1}}\overline{Z_2},$$
  $$\overline{\tau(Z_1,Z_2)}=\tau(\overline{Z_1},\overline{Z_2}),{\quad}\overline{R(Z_1,Z_2)Z_3}=R(\overline{Z_1},\overline{Z_2})\overline{Z_3},{\quad}\overline{Q(Z_1,Z_2)}=Q(\overline{Z_1},\overline{Z_2}),$$
  for any $Z_1,Z_2,Z_3\in{\mathbb{C}HM}$. The complex conjugation can be reflected in the indices of the components of $\omega_a^b$, $h_{ab}$, $J_{\ a}^b$, $A_{ab}$, $R_{abcd}$ and their covariant derivatives, e.g.,
$$\overline{\omega_{\alpha}^{\bar{\beta}}}=\omega_{\bar{\alpha}}^{\beta},{\quad}\overline{J_{\ \alpha}^{\beta}}=J_{\ \bar{\alpha}}^{\bar{\beta}},{\quad}\overline{h_{\alpha\bar{\beta}}}=h_{\bar{\alpha}\beta}.$$
\end{rem}

\section{The second- and third-order covariant derivatives and their commutation formulae}
\subsection{The second- and third-order covariant derivatives}
  The second-order covariant derivative of $u$ is defined as
  \begin{equation}\label{eq3.1}
  \nabla^2u(X,Y):=X(Yu)-(\nabla_XY)u,{\quad}u_{jk}:=\nabla^2u(W_j,W_k),
  \end{equation}
  for any vector fields $X$, $Y$, and the third-order covariant derivative of $u$ is defined as
  \begin{align}\label{eq3.2}
  \notag &\nabla^3u(X,Y,Z)=\big(\nabla_X\nabla^2u\big)(Y,Z)=X(\nabla^2u(Y,Z))-\nabla^2u(\nabla_XY,Z)-\nabla^2u(Y,\nabla_XZ),\\
  &u_{jkl}:=\nabla^3u(W_j,W_k,W_l).
  \end{align}
  for any vector fields $X$, $Y$, $Z$. By \eqref{eq3.1}, for the second-order covariant derivative, we have
  $$u_{jk}=\nabla^2u(W_j,W_k)=W_jW_ku-(\nabla_{W_j}W_k)u=W_j(u_k)-\Gamma_{jk}^lu_l.$$
  In particular, by the vanishing of connection coefficients in \eqref{QBD} we get
  \begin{equation}\label{SD}
  \begin{aligned}
  &u_{\alpha\lambda}=\nabla^2u(W_{\alpha},W_{\lambda})=W_{\alpha}(u_{\lambda})-\Gamma_{\alpha\lambda}^{\beta}u_{\beta}-\Gamma_{\alpha\lambda}^{\bar{\beta}}u_{\bar{\beta}}
  =W_{\alpha}(u_{\lambda})-\Gamma_{\alpha\lambda}^{\beta}u_{\beta}+\frac{i}{2}Q_{\lambda\alpha}^{\bar{\beta}}u_{\bar{\beta}},\\
  &u_{\alpha\bar{\lambda}}=W_{\alpha}(u_{\bar{\lambda}})-\Gamma_{\alpha\bar{\lambda}}^{\bar{\beta}}u_{\bar{\beta}},\\
  &u_{\alpha0}=W_{\alpha}(u_0),{\quad}u_{0\alpha}=T(u_{\alpha})-\Gamma_{0\alpha}^{\beta}u_{\beta}.
  \end{aligned}
  \end{equation}

In the following, the vanishing of connection coefficients in \eqref{QBD}, especially,
$$\Gamma_{\alpha\bar{\beta}}^{\gamma}=\Gamma_{\bar{\alpha}\beta}^{\bar{\gamma}}=0,$$
will be used frequently. By \eqref{eq3.2}, for the third-order covariant derivative, we have
\begin{equation}
\notag u_{abc}=W_a(u_{bc})-\Gamma_{ab}^du_{dc}-\Gamma_{ac}^du_{bd}.
\end{equation}
In particular, by \eqref{QBD}, we have
\begin{equation}\label{TD}
\begin{aligned}
&u_{\alpha\beta\gamma}=W_{\alpha}(u_{\beta\gamma})-\Gamma_{\alpha\beta}^{\mu}u_{\mu\gamma}-\Gamma_{\alpha\beta}^{\bar{\mu}}u_{\bar{\mu}\gamma}
-\Gamma_{\alpha\gamma}^{\mu}u_{\beta\mu}-\Gamma_{\alpha\gamma}^{\bar{\mu}}u_{\beta\bar{\mu}},\\
&u_{\bar{\alpha}\beta\gamma}=W_{\bar{\alpha}}(u_{\beta\gamma})
-\Gamma_{\bar{\alpha}\beta}^{\mu}u_{\mu\gamma}-\Gamma_{\bar{\alpha}\gamma}^{\mu}u_{\beta\mu},\\
&u_{\alpha\bar{\beta}\gamma}=W_{\alpha}(u_{\bar{\beta}\gamma})-\Gamma_{\alpha\bar{\beta}}^{\bar{\mu}}u_{\bar{\mu}\gamma}
-\Gamma_{\alpha\gamma}^{\mu}u_{\bar{\beta}\mu}-\Gamma_{\alpha\gamma}^{\bar{\mu}}u_{\bar{\beta}\bar{\mu}},\\
&u_{\bar{\alpha}\bar{\beta}\gamma}=W_{\bar{\alpha}}(u_{\bar{\beta}\gamma})-\Gamma_{\bar{\alpha}\bar{\beta}}^{\mu}u_{\mu\gamma}
-\Gamma_{\bar{\alpha}\bar{\beta}}^{\bar{\mu}}u_{\bar{\mu}\gamma}-\Gamma_{\bar{\alpha}\gamma}^{\mu}u_{\bar{\beta}\mu}.
\end{aligned}
\end{equation}
\begin{rem}
(1) In \eqref{SD} and \eqref{TD}, we have used $\Gamma_{ab}^0=0$, $\Gamma_{\alpha\bar{\beta}}^{\gamma}=0$ and $\Gamma_{\alpha\beta}^{\bar{\gamma}}=-\frac{i}{2}Q_{\beta\alpha}^{\bar{\gamma}}$ repeatedly.

(2) The complex conjugation can also be reflected in the indices of the components of any-order covariant derivative of a real function $u$, e.g. $\overline{u_{\alpha\beta}}=u_{\bar{\alpha}\bar{\beta}},{\quad}
\overline{u_{\alpha\bar{\beta}\gamma}}=u_{\bar{\alpha}\beta\bar{\gamma}}.$
\end{rem}
\subsection{The sub-Laplacian}
On a contact Riemannian manifold $M$, with respect to a local $T^{(1,0)}M$-frame $\{W_j\}$, we define the sub-Laplacian operator as
$$\Delta_{b}u=u_{\phantom{\alpha}\alpha}^{\alpha}+u_{\phantom{\bar{\alpha}}\bar{\alpha}}^{\bar{\alpha}},$$
for $u\in{C_0^{\infty}(M)}$. Furthermore, if $\{W_j\}$ is an orthonormal $T^{(1,0)}M$-frame, we have
\begin{equation}\label{DSL}
\Delta_{b}u=u_{\alpha\bar{\alpha}}+u_{\bar{\alpha}\alpha}.
\end{equation}
For any functions $u\in{C_0^{\infty}(M)}$ and $v{\in}C^{\infty}(M)$, we define the $L^2$ inner product $(\cdot,\cdot)$ as
\begin{equation}\label{eq3.5a}
(u,v)=\int_Mu\bar{v}dV.
\end{equation}
For any vector field $X$, $X^{\ast}$ is called the {\it formal adjoint} of $X$ if $(Xu,v)=(u,X^{\ast}v)$ for $u,v{\in}C_0^{\infty}M$. And $\Delta_{b}$ is hypoelliptic and by a result of \cite{MS1} has a discrete spectrum
\begin{equation*}
0<\lambda_1<\lambda_2<\cdots<{\ }\uparrow+\infty.
\end{equation*}
\begin{lem}\label{lem3.1}
We have
\begin{equation}\label{FA}
W_{\alpha}^{\ast}=-W_{\bar{\alpha}}+\Gamma_{\bar{\beta}\beta}^{\alpha},{\qquad}
W_{\bar{\alpha}}^{\ast}=-W_{\alpha}+\Gamma_{\beta\bar{\beta}}^{\bar{\alpha}},{\qquad}(iT)^{\ast}=iT.
\end{equation}
\end{lem}
\begin{proof}
By \eqref{SE}, we have
\begin{align}
\notag d\theta^n&=(-2i\theta^{\alpha}\wedge\theta^{\bar{\alpha}})^n
=(-2)^ni^nn!\theta^1\wedge\theta^{\bar1}\wedge\cdots\wedge\theta^n\wedge\theta^{\bar{n}}\\
\notag &=(-2)^ni^nn!(-1)^{n(n-1)/2}\theta^1\wedge\theta^2\wedge\cdots\wedge\theta^n\wedge\theta^{\bar1}\wedge\cdots\wedge\theta^{\bar{n}}\\
\notag &=(-2)^ni^{n^2}n!\theta^1\wedge\theta^2\wedge\cdots\wedge\theta^n\wedge\theta^{\bar1}\wedge\cdots\wedge\theta^{\bar{n}}.
\end{align}
So the volume form is
\begin{equation}\label{vf}
dV:=\theta{\wedge}d\theta^n=(-2)^ni^{n^2}n!\theta\wedge\theta^1\wedge\theta^2\wedge\cdots\wedge\theta^n\wedge\theta^{\bar1}\wedge\cdots\wedge\theta^{\bar{n}}.
\end{equation}
For any vector field $X$ and $u\in{C_0^{\infty}(M)}$, we get
\begin{align}\label{eq3.51}
\notag \int_MXuvdV&=\int_Mvdu{\wedge}i_{X}dV=-\int_Mudv{\wedge}i_{X}dV-\int_Muvd(i_{X}dV)+\int_Md(uvi_XdV)\\
&=-\int_MuXvdV-\int_Muvd(i_{X}dV),
\end{align}
by Stokes' formula and $0=\int_Mi_{X}(vdu\wedge{dV})=\int_MvXudV-\int_Mvdu{\wedge}i_{X}dV$.
It follows from the structure equation \eqref{SE} that
\begin{equation}\label{eq3.5b}
d\theta^{\beta}=\theta^{\gamma}\wedge\omega_{\gamma}^{\beta}+\theta^{\bar{\gamma}}\wedge\omega_{\bar{\gamma}}^{\beta}
=\Gamma_{\mu\gamma}^{\beta}\theta^{\gamma}\wedge\theta^{\mu}+\Gamma_{\bar{\mu}\gamma}^{\beta}\theta^{\gamma}\wedge\theta^{\bar{\mu}}+\Gamma_{\mu\bar{\gamma}}^{\beta}\theta^{\bar{\gamma}}\wedge\theta^{\mu}+\Gamma_{\bar{\mu}\bar{\gamma}}^{\beta}\theta^{\bar{\gamma}}\wedge\theta^{\bar{\mu}},{\quad}mod{\quad}\theta.
\end{equation}
Applying \eqref{vf} and \eqref{eq3.5b}, we get
\begin{align}
\notag d(i_{W_{\alpha}}dV)&=(-2)^ni^{n^2}n!d\bigg((-1)^{\alpha}\theta\wedge\theta^1\wedge\cdots\wedge\widehat{\theta^{\alpha}}\wedge\cdots
\wedge\theta^n\wedge\cdots\wedge\theta^{\bar{n}}\bigg)\\
\notag &=(-1)^{\alpha}(-2)^ni^{n^2}n!\bigg(\sum\limits_{\beta<\alpha}(-1)^{\beta}\theta\wedge\theta^1\wedge\cdots
\wedge{d\theta^{\beta}}\wedge\cdots\wedge\widehat{\theta^{\alpha}}\wedge\cdots\wedge\theta^n\wedge\cdots\wedge\theta^{\bar{n}}\\
\notag &\ \ \ +\sum\limits_{\beta>\alpha}(-1)^{\beta-1}\theta\wedge\theta^1\wedge\cdots\wedge\widehat{\theta^{\alpha}}\cdots
\wedge{d\theta^{\beta}}\wedge\cdots\wedge\theta^n\wedge\cdots\wedge\theta^{\bar{n}}\\
\notag &\ \ \ +\sum\limits_{\beta=1}^n(-1)^{n+\beta-1}\theta\wedge\theta^1\wedge\cdots\wedge\widehat{\theta^{\alpha}}\cdots\wedge\theta^n\wedge
\cdots\wedge{d\theta^{\bar{\beta}}}\wedge\cdots\wedge\theta^{\bar{n}}\bigg)\\
\notag &=(-2)^ni^{n^2}n!\bigg(-\Gamma_{\alpha\beta}^{\beta}+\Gamma_{\beta\alpha}^{\beta}-\Gamma_{\alpha\bar{\beta}}^{\bar{\beta}}
+\Gamma_{\bar{\beta}\alpha}^{\bar{\beta}}\bigg)\theta\wedge\theta^1\wedge\cdots\wedge\theta^{\bar{n}}\\
\notag &=\bigg(-\Gamma_{\alpha\beta}^{\beta}+\Gamma_{\beta\alpha}^{\beta}-\Gamma_{\alpha\bar{\beta}}^{\bar{\beta}}
+\Gamma_{\bar{\beta}\alpha}^{\bar{\beta}}\bigg)dV=\Gamma_{\beta\alpha}^{\beta}dV.
\end{align}
The last identity holds because $\Gamma_{\alpha\beta}^{\beta}+\Gamma_{\alpha\bar{\beta}}^{\bar{\beta}}=0$ by \eqref{eq2.9} and $\Gamma_{\bar{\beta}\alpha}^{\bar{\beta}}=0$ by \eqref{QBD}. For any $u{\in}C_0^{\infty}(M)$ and $v{\in}C^{\infty}(M)$, apply \eqref{eq3.51} with $X=W_{\alpha}$ to get
\begin{align}\label{eq3.52}
\notag (W_{\alpha}u,v)&=\int_MW_{\alpha}u\bar{v}dV=-\int_MuW_{\alpha}\bar{v}dV-\int_Mu\bar{v}d(i_{W_{\alpha}}dV)
=\int_Mu\big(-W_{\alpha}\bar{v}-\Gamma_{\beta\alpha}^{\beta}\bar{v}\big)dV\\
&=\int_Mu\overline{\big(-W_{\bar{\alpha}}v+\Gamma_{\bar{\beta}\beta}^{\alpha}v\big)}dV
=\Bigg(u,\bigg(-W_{\bar{\alpha}}+\Gamma_{\bar{\beta}\beta}^{\alpha}\bigg)v\Bigg),
\end{align}
by $\Gamma_{\beta\alpha}^{\beta}=-\Gamma_{\beta\bar{\beta}}^{\bar{\alpha}}$ in \eqref{eq2.9}.  The first identity in \eqref{FA} holds. The second identity in \eqref{FA} follows from taking conjugation.

By the structure equation \eqref{SE}, $d\theta^{\alpha}$ doesn't contain $\theta\wedge\theta^{\alpha}$ terms and $d\theta^{\bar{\alpha}}$ doesn't contain $\theta\wedge\theta^{\bar{\alpha}}$ terms. So
\begin{equation}\label{eq3.53}
d(i_TdV)=(-2)^ni^{n^2}n!d\bigg(\theta^1\wedge\theta^2\wedge\cdots\wedge\theta^n\wedge\theta^{\bar1}\wedge\cdots\wedge\theta^{\bar{n}}\bigg)=0.
\end{equation}
Apply \eqref{eq3.51} with $X=T$ to get
$$(iTu,v)=i\int_MTu\bar{v}dV=-i\int_MuT\bar{v}dV=(u,iTv).$$
$(iT)^{\ast}=iT$ follows.
\end{proof}
\begin{cor}\label{cor3.1}
With respect to an orthonormal${\ \ }$$T^{(1,0)}M$-frame $\{W_j\}$, we have
$$(\Delta_{b}u,v)=-\sum\limits_{\alpha}(u_{\alpha},v_{\alpha})+(u_{\bar{\alpha}},v_{\bar{\alpha}}),$$
for $u{\in}C_0^{\infty}(M)$ and $v{\in}C^{\infty}(M)$
\end{cor}
\begin{proof}
By the definition of $\Delta_{b}$, we get
\begin{align}
\notag (\Delta_{b}u,v)&=\sum\limits_{\alpha}(u_{\bar{\alpha}\alpha}+u_{\alpha\bar{\alpha}},v)
=\sum\limits_{\alpha,\beta}\bigg(W_{\bar{\alpha}}u_{\alpha}-\Gamma_{\bar{\alpha}\alpha}^{\beta}u_{\beta},v\bigg)
+\bigg(W_{\alpha}u_{\bar{\alpha}}-\Gamma_{\alpha\bar{\alpha}}^{\bar{\beta}}u_{\bar{\beta}},v\bigg)\\
\notag &=\sum\limits_{\alpha,\beta}\bigg(u_{\alpha},W_{\bar{\alpha}}^{\ast}v\bigg)-\bigg(\Gamma_{\bar{\alpha}\alpha}^{\beta}u_{\beta},v\bigg)
+\bigg(u_{\bar{\alpha}},W_{\alpha}^{\ast}v\bigg)-\bigg(\Gamma_{\alpha\bar{\alpha}}^{\bar{\beta}}u_{\bar{\beta}},v\bigg)\\
\notag &=-\sum\limits_{\alpha,\beta}(u_{\alpha},v_{\alpha})+\bigg(u_{\alpha},\Gamma_{\beta\bar{\beta}}^{\bar{\alpha}}v\bigg)
-\bigg(\Gamma_{\bar{\alpha}\alpha}^{\beta}u_{\beta},v\bigg)-(u_{\bar{\alpha}},v_{\bar{\alpha}})
+\bigg(u_{\bar{\alpha}},\Gamma_{\bar{\beta}\beta}^{\alpha}v\bigg)-\bigg(\Gamma_{\alpha\bar{\alpha}}^{\bar{\beta}}u_{\bar{\beta}},v\bigg)\\
\notag &=-\sum\limits_{\alpha}(u_{\alpha},v_{\alpha})+(u_{\bar{\alpha}},v_{\bar{\alpha}}).
\end{align}
\end{proof}
\begin{cor}\label{cor3.2}
\begin{equation}\label{eq6.1a}
(\Delta_{b}u)_{\alpha}=u_{\alpha\beta\bar{\beta}}+u_{\alpha\bar{\beta}\beta}.
\end{equation}
\end{cor}
\begin{proof}
By \eqref{TD} and Corollary \ref{cor3.1}, we get
\begin{align}
\notag u_{\alpha\beta\bar{\beta}}+&u_{\alpha\bar{\beta}\beta}=W_{\alpha}(u_{\beta\bar{\beta}})-\Gamma_{\alpha\beta}^{\mu}u_{\mu\bar{\beta}}
-\Gamma_{\alpha\beta}^{\bar{\mu}}u_{\bar{\mu}\bar{\beta}}-\Gamma_{\alpha\bar{\beta}}^{\bar{\mu}}u_{\beta\bar{\mu}}
+W_{\alpha}(u_{\bar{\beta}\beta})-\Gamma_{\alpha\bar{\beta}}^{\bar{\mu}}u_{\bar{\mu}\beta}
-\Gamma_{\alpha\beta}^{\mu}u_{\bar{\beta}\mu}-\Gamma_{\alpha\beta}^{\bar{\mu}}u_{\bar{\beta}\bar{\mu}}\\
\notag &=W_{\alpha}(u_{\beta\bar{\beta}}+u_{\bar{\beta}\beta})
-\bigg(\Gamma_{\alpha\beta}^{\mu}u_{\mu\bar{\beta}}+\Gamma_{\alpha\bar{\beta}}^{\bar{\mu}}u_{\beta\bar{\mu}}\bigg)
-\bigg(\Gamma_{\alpha\bar{\beta}}^{\bar{\mu}}u_{\bar{\mu}\beta}+\Gamma_{\alpha\beta}^{\mu}u_{\bar{\beta}\mu}\bigg)
-\bigg(\Gamma_{\alpha\beta}^{\bar{\mu}}u_{\bar{\mu}\bar{\beta}}+\Gamma_{\alpha\beta}^{\bar{\mu}}u_{\bar{\beta}\bar{\mu}}\bigg)\\
\notag &=W_{\alpha}(u_{\beta\bar{\beta}}+u_{\bar{\beta}\beta})=(\Delta_{b}u)_{\alpha},
\end{align}
by the first two identities in \eqref{eq2.9}.
\end{proof}

\subsection{The commutation formulae}
\begin{pro}\label{pro3.1}
For the second-order covariant derivatives of a function $u$, we have the following commutation formulae.
\begin{equation}
\notag u_{ij}-u_{ji}=-\tau(W_i,W_j)u.
\end{equation}
\end{pro}
\begin{proof}
By the definition \eqref{eq3.1}, we get $u_{ij}-u_{ji}=W_iW_ju-\nabla_{W_i}W_ju-W_jW_iu+\nabla_{W_j}W_iu
=-\big(\nabla_{W_i}W_j-\nabla_{W_j}W_i-[W_i,W_j]\big)u=-\tau(W_i,W_j)u.$
\end{proof}
In particular, Proposition \ref{pro3.1} and \eqref{TWT} implies that
\begin{equation}\label{SDCF}
\begin{aligned}
&u_{\alpha\bar{\beta}}-u_{\bar{\beta}\alpha}=-\tau(W_{\alpha},W_{\bar{\beta}})u
=-2d\theta(W_{\alpha},W_{\bar{\beta}})Tu=2ih_{\alpha\bar{\beta}}u_0,\\
&u_{\alpha\beta}=u_{\beta\alpha},\\
&u_{0\alpha}-u_{\alpha0}=-\tau(T,W_{\alpha})=-A_{\alpha}^{\bar{\beta}}u_{\bar{\beta}}.
\end{aligned}
\end{equation}
Following Proposition 9.2 and 9.3 in \cite{BD1} in the pseudohermitian case, we call the relations between the third-order covariant derivatives of functions $u_{abc}$ and $u_{acb}$ {\it the inner commutation formulae} and the relations between $u_{abc}$ and $u_{bac}$  {\it the outer commutation formulae}.
\begin{pro}
(1) We have the inner commutation formulae
\begin{equation}\label{ICF}
\begin{aligned}
&u_{\bar{\alpha}\beta\gamma}=u_{\bar{\alpha}\gamma\beta},\\
&u_{\alpha\bar{\beta}\gamma}=u_{\alpha\gamma\bar{\beta}}-2ih_{\gamma\bar{\beta}}u_{\alpha0}.
\end{aligned}
\end{equation}
(2) We have the outer commutation formulae
\begin{equation}\label{OCF}
\begin{aligned}
&u_{\bar{\rho}\gamma\alpha}=u_{\gamma\bar{\rho}\alpha}-2ih_{\gamma\bar{\rho}}u_{0\alpha}+u_{\beta}R_{\alpha\ \gamma\bar{\rho}}^{\ \beta}+\frac{i}{2}Q_{\alpha\gamma,\bar{\rho}}^{\bar{\beta}}u_{\bar{\beta}},\\
&u_{\bar{\rho}\bar{\gamma}\alpha}=u_{\bar{\gamma}\bar{\rho}\alpha}+2i\bigg(A_{\bar{\gamma}}^{\beta}h_{\alpha\bar{\rho}}-A_{\bar{\rho}}^{\beta}h_{\alpha\bar{\gamma}}\bigg)u_{\beta}
-\frac{i}{2}Q_{\bar{\rho}\bar{\beta},\alpha}^{\gamma}u_{\beta}-\frac{1}{4}Q_{\alpha\beta}^{\bar{\mu}}Q_{\bar{\rho}\bar{\beta}}^{\gamma}u_{\bar{\mu}}.
\end{aligned}
\end{equation}
\end{pro}
\begin{proof}
(1) The first identity of \eqref{ICF} follows directly from the second identity in \eqref{TD} and $u_{\alpha\beta}=u_{\beta\alpha}$ in \eqref{SDCF}.

Taking conjugation of the fourth identity in \eqref{TD}, we get $$u_{\alpha\gamma\bar{\beta}}=W_{\alpha}(u_{\gamma\bar{\beta}})-\Gamma_{\alpha\gamma}^{\mu}u_{\mu\bar{\beta}}
-\Gamma_{\alpha\gamma}^{\bar{\mu}}u_{\bar{\mu}\bar{\beta}}-\Gamma_{\alpha\bar{\beta}}^{\bar{\mu}}u_{\gamma\bar{\mu}}.$$
So by \eqref{eq2.9}, \eqref{TD} and \eqref{SDCF}, we get
\begin{align}
\notag u_{\alpha\bar{\beta}\gamma}&=W_{\alpha}(u_{\bar{\beta}\gamma})-\Gamma_{\alpha\bar{\beta}}^{\bar{\mu}}u_{\bar{\mu}\gamma}
-\Gamma_{\alpha\gamma}^{\mu}u_{\bar{\beta}\mu}-\Gamma_{\alpha\gamma}^{\bar{\mu}}u_{\bar{\beta}\bar{\mu}}\\
\notag &=W_{\alpha}\bigg(u_{\gamma\bar{\beta}}-2ih_{\gamma\bar{\beta}}u_0\bigg)
-\Gamma_{\alpha\bar{\beta}}^{\bar{\mu}}\bigg(u_{\gamma\bar{\mu}}-2ih_{\gamma\bar{\mu}}u_0\bigg)
-\Gamma_{\alpha\gamma}^{\mu}\bigg(u_{\mu\bar{\beta}}-2ih_{\mu\bar{\beta}}u_0\bigg)
-\Gamma_{\alpha\gamma}^{\bar{\mu}}u_{\bar{\mu}\bar{\beta}}\\
\notag &=W_{\alpha}(u_{\gamma\bar{\beta}})-\Gamma_{\alpha\gamma}^{\mu}u_{\mu\bar{\beta}}
-\Gamma_{\alpha\gamma}^{\bar{\mu}}u_{\bar{\mu}\bar{\beta}}-\Gamma_{\alpha\bar{\beta}}^{\bar{\mu}}u_{\gamma\bar{\mu}}
-2i\bigg(h_{\gamma\bar{\beta}}W_{\alpha}(u_0)-\Gamma_{\alpha\bar{\beta}}^{\bar{\gamma}}u_0-\Gamma_{\alpha\gamma}^{\beta}u_0\bigg)\\
&=u_{\alpha\gamma\bar{\beta}}-2ih_{\gamma\bar{\beta}}u_{\alpha0}.
\end{align}
The second identity of \eqref{ICF} is proved.

(2) For the first identity in \eqref{OCF}, note that
\begin{equation}\label{eq3.17}
\begin{aligned}
du_{\alpha}&=(W_{\beta}u_{\alpha})\theta^{\beta}+(W_{\bar{\beta}}u_{\alpha})\theta^{\bar{\beta}}+T(u_{\alpha})\theta\\
&=\bigg(u_{\beta\alpha}+\Gamma_{\beta\alpha}^{\rho}u_{\rho}-\frac{i}{2}Q_{\alpha\beta}^{\bar{\rho}}u_{\bar{\rho}}\bigg)\theta^{\beta}
+(u_{\bar{\beta}\alpha}+\Gamma_{\bar{\beta}\alpha}^{\rho}u_{\rho})\theta^{\bar{\beta}}+(u_{0\alpha}+\Gamma_{0\alpha}^{\rho}u_{\rho})\theta\\ &=u_{\beta\alpha}\theta^{\beta}+u_{\bar{\beta}\alpha}\theta^{\bar{\beta}}+u_{0\alpha}\theta+u_{\beta}\omega_{\alpha}^{\beta}-\frac{i}{2}Q_{\alpha\beta}^{\bar{\rho}}u_{\bar{\rho}}\theta^{\beta}
\end{aligned}
\end{equation}
by using \eqref{SD} and $\omega_{\alpha}^{\rho}=\Gamma_{\beta\alpha}^{\rho}\theta^{\beta}+\Gamma_{\bar{\beta}\alpha}^{\rho}\theta^{\bar{\beta}}+\Gamma_{0\alpha}^{\rho}\theta$.
Taking exterior differentiation on both sides of \eqref{eq3.17}, we get
\begin{align}\label{eq3.18}
\notag 0&=du_{\beta\alpha}\wedge\theta^{\beta}+du_{\bar{\beta}\alpha}\wedge\theta^{\bar{\beta}}+du_{0\alpha}\wedge\theta
+u_{\beta\alpha}d\theta^{\beta}+u_{\bar{\beta}\alpha}d\theta^{\bar{\beta}}+u_{0\alpha}d\theta\\
&\ \ \ +du_{\beta}\wedge\omega_{\alpha}^{\beta}+u_{\beta}d\omega_{\alpha}^{\beta}
-\frac{i}{2}d(Q_{\alpha\beta}^{\bar{\rho}}u_{\bar{\rho}}){\wedge}\theta^{\beta}
-\frac{i}{2}Q_{\alpha\beta}^{\bar{\rho}}u_{\bar{\rho}}d\theta^{\beta}.
\end{align}
Note that
\begin{equation}\label{eq3.22}
\begin{aligned}
du_{\beta\alpha}&=W_{\gamma}(u_{\beta\alpha})\theta^{\gamma}+W_{\bar{\gamma}}(u_{\beta\alpha})\theta^{\bar{\gamma}}+T(u_{\beta\alpha})\theta,\\
\omega_{\alpha}^{\beta}&=\Gamma_{\gamma\alpha}^{\beta}\theta^{\gamma}+\Gamma_{\bar{\gamma}\alpha}^{\beta}\theta^{\bar{\gamma}},{\quad}mod{\quad}{\theta},\\
d\theta^{\beta}&=\theta^{\gamma}\wedge\omega_{\gamma}^{\beta}+\theta^{\bar{\gamma}}\wedge\omega_{\bar{\gamma}}^{\beta}
=\Gamma_{a\gamma}^{\beta}\theta^{\gamma}\wedge\theta^a+\Gamma_{a\bar{\gamma}}^{\beta}\theta^{\bar{\gamma}}\wedge\theta^a
=\Gamma_{\bar{\rho}\gamma}^{\beta}\theta^{\gamma}\wedge\theta^{\bar{\rho}}+\Gamma_{\bar{\rho}\bar{\gamma}}^{\beta}\theta^{\bar{\gamma}}\wedge\theta^{\bar{\rho}},
{\quad}mod{\quad}\theta,\ \theta^{\gamma}\wedge\theta^{\rho},\\
d\theta^{\bar{\beta}}&=\theta^{\gamma}\wedge\omega_{\gamma}^{\bar{\beta}}+\theta^{\bar{\gamma}}\wedge\omega_{\bar{\gamma}}^{\bar{\beta}}
=\Gamma_{a\gamma}^{\bar{\beta}}\theta^{\gamma}\wedge\theta^a+\Gamma_{a\bar{\gamma}}^{\bar{\beta}}\theta^{\bar{\gamma}}\wedge\theta^a
=-\Gamma_{\gamma\bar{\rho}}^{\bar{\beta}}\theta^{\gamma}\wedge\theta^{\bar{\rho}}+\Gamma_{\bar{\rho}\bar{\gamma}}^{\bar{\beta}}\theta^{\bar{\gamma}}\wedge\theta^{\bar{\rho}},
{\ }mod{\quad}\theta,\ \theta^{\gamma}\wedge\theta^{\rho},\\ d\omega_{\alpha}^{\beta}&=\omega_{\alpha}^{\mu}\wedge\omega_{\mu}^{\beta}
+\omega_{\alpha}^{\bar{\mu}}\wedge\omega_{\bar{\mu}}^{\beta}
+R_{\alpha\ \lambda\bar{\mu}}^{\ \beta}\theta^{\lambda}\wedge\theta^{\bar{\mu}}+\frac{1}{2}R_{\alpha\ \bar{\lambda}\bar{\mu}}^{\ \beta}\theta^{\bar{\lambda}}\wedge\theta^{\bar{\mu}}\\
&=\Gamma_{\gamma\alpha}^{\mu}\Gamma_{\bar{\rho}\mu}^{\beta}\theta^{\gamma}\wedge\theta^{\bar{\rho}}
+\Gamma_{\bar{\gamma}\alpha}^{\mu}\Gamma_{\rho\mu}^{\beta}\theta^{\bar{\gamma}}\wedge\theta^{\rho}
+\Gamma_{\gamma\alpha}^{\bar{\mu}}\Gamma_{\bar{\rho}\bar{\mu}}^{\beta}\theta^{\gamma}\wedge\theta^{\bar{\rho}}
+R_{\alpha\ \gamma\bar{\rho}}^{\ \beta}\theta^{\gamma}\wedge\theta^{\bar{\rho}}\\
&\ \ \ +\Gamma_{\bar{\gamma}\alpha}^{\mu}\Gamma_{\bar{\rho}\mu}^{\beta}\theta^{\bar{\gamma}}\wedge\theta^{\bar{\rho}}
+\frac{1}{2}R_{\alpha\ \bar{\gamma}\bar{\rho}}^{\ \beta}\theta^{\bar{\gamma}}\wedge\theta^{\bar{\rho}},{\qquad}{\qquad}{\qquad}{\qquad}{\qquad}{\qquad}{\qquad}{\qquad}{\qquad}mod{\quad}\theta,\ \theta^{\gamma}\wedge\theta^{\rho},
\end{aligned}
\end{equation}
by \eqref{SE} and $\Gamma_{\alpha\bar{\beta}}^{\gamma}=0$ in \eqref{QBD}. Substituting \eqref{eq3.22} to the corresponding terms in \eqref{eq3.18} and comparing the coefficients of $\theta^{\gamma}\wedge\theta^{\bar{\rho}}$, we get
\begin{align}\label{eq3.27}
\notag 0&=-W_{\bar{\rho}}u_{\gamma\alpha}+W_{\gamma}u_{\bar{\rho}\alpha}+u_{\beta\alpha}\Gamma_{\bar{\rho}\gamma}^{\beta}-u_{\bar{\beta}\alpha}\Gamma_{\gamma\bar{\rho}}^{\bar{\beta}}
-2ih_{\gamma\bar{\rho}}u_{0\alpha}+W_{\gamma}(u_{\beta})\Gamma_{\bar{\rho}\alpha}^{\beta}-W_{\bar{\rho}}(u_{\beta})\Gamma_{\gamma\alpha}^{\beta}\\
\notag &\ \ \ +u_{\beta}\bigg(R_{\alpha\ \gamma\bar{\rho}}^{\ \beta}+\Gamma_{\gamma\alpha}^{\mu}\Gamma_{\bar{\rho}\mu}^{\beta}-\Gamma_{\bar{\rho}\alpha}^{\mu}\Gamma_{\gamma\mu}^{\beta}
+\Gamma_{\gamma\alpha}^{\bar{\mu}}\Gamma_{\bar{\rho}\bar{\mu}}^{\beta}\bigg)
+\frac{i}{2}W_{\bar{\rho}}\bigg(Q_{\alpha\gamma}^{\bar{\beta}}u_{\bar{\beta}}\bigg)
-\frac{i}{2}Q_{\alpha\beta}^{\bar{\mu}}u_{\bar{\mu}}\Gamma_{\bar{\rho}\gamma}^{\beta}\\
\notag &=-W_{\bar{\rho}}u_{\gamma\alpha}+W_{\gamma}u_{\bar{\rho}\alpha}+u_{\beta\alpha}\Gamma_{\bar{\rho}\gamma}^{\beta}
-u_{\bar{\beta}\alpha}\Gamma_{\gamma\bar{\rho}}^{\bar{\beta}}-2ih_{\gamma\bar{\rho}}u_{0\alpha}+u_{\beta}R_{\alpha\ \gamma\bar{\rho}}^{\ \beta}\\
\notag &\ \ \ +\bigg(W_{\gamma}(u_{\beta})\Gamma_{\bar{\rho}\alpha}^{\beta}-u_{\mu}\Gamma_{\gamma\beta}^{\mu}\Gamma_{\bar{\rho}\alpha}^{\beta}
+\frac{i}{2}u_{\bar{\mu}}Q_{\beta\gamma}^{\bar{\mu}}\Gamma_{\bar{\rho}\alpha}^{\beta}\bigg)
-\frac{i}{2}u_{\bar{\mu}}Q_{\beta\gamma}^{\bar{\mu}}\Gamma_{\bar{\rho}\alpha}^{\beta}
-\bigg(W_{\bar{\rho}}(u_{\beta})\Gamma_{\gamma\alpha}^{\beta}-u_{\mu}\Gamma_{\bar{\rho}\beta}^{\mu}\Gamma_{\gamma\alpha}^{\beta}\bigg)\\
&\ \ \ -\frac{i}{2}u_{\beta}Q_{\alpha\gamma}^{\bar{\mu}}\Gamma_{\bar{\rho}\bar{\mu}}^{\beta}
+\frac{i}{2}W_{\bar{\rho}}\bigg(Q_{\alpha\gamma}^{\bar{\beta}}u_{\bar{\beta}}\bigg)
-\frac{i}{2}Q_{\alpha\beta}^{\bar{\mu}}u_{\bar{\mu}}\Gamma_{\bar{\rho}\gamma}^{\beta}.
\end{align}
Substitute
\begin{equation}
\begin{aligned}
\notag &W_{\gamma}(u_{\beta})-\Gamma_{\gamma\beta}^{\mu}u_{\mu}+\frac{i}{2}Q_{\beta\gamma}^{\bar{\mu}}u_{\bar{\mu}}=u_{\gamma\beta},
{\qquad}W_{\bar{\rho}}(u_{\beta})-\Gamma_{\bar{\rho}\beta}^{\mu}u_{\mu}=u_{\bar{\rho}\beta},
\end{aligned}
\end{equation}
by \eqref{SD}, into two brackets in \eqref{eq3.27} to get
\begin{align}
\notag 0&=-W_{\bar{\rho}}u_{\gamma\alpha}+W_{\gamma}u_{\bar{\rho}\alpha}+u_{\beta\alpha}\Gamma_{\bar{\rho}\gamma}^{\beta}-u_{\bar{\beta}\alpha}\Gamma_{\gamma\bar{\rho}}^{\bar{\beta}}
-2ih_{\gamma\bar{\rho}}u_{0\alpha}+u_{\beta}R_{\alpha\ \gamma\bar{\rho}}^{\ \beta}+u_{\gamma\beta}\Gamma_{\bar{\rho}\alpha}^{\beta}-u_{\bar{\rho}\beta}\Gamma_{\gamma\alpha}^{\beta}\\
\notag &\ \ \ -\frac{i}{2}u_{\beta}Q_{\alpha\gamma}^{\bar{\mu}}\Gamma_{\bar{\rho}\bar{\mu}}^{\beta}-\frac{i}{2}Q_{\beta\gamma}^{\bar{\mu}}u_{\bar{\mu}}\Gamma_{\bar{\rho}\alpha}^{\beta}
+\frac{i}{2}W_{\bar{\rho}}(Q_{\alpha\gamma}^{\bar{\beta}})u_{\bar{\beta}}+\frac{i}{2}Q_{\alpha\gamma}^{\bar{\beta}}(W_{\bar{\rho}}u_{\bar{\beta}})-\frac{i}{2}Q_{\alpha\beta}^{\bar{\mu}}u_{\bar{\mu}}\Gamma_{\bar{\rho}\gamma}^{\beta}\\
\notag &=\bigg(-W_{\bar{\rho}}u_{\gamma\alpha}+\Gamma_{\bar{\rho}\gamma}^{\beta}u_{\beta\alpha}+\Gamma_{\bar{\rho}\alpha}^{\beta}u_{\gamma\beta}\bigg)
+\bigg(W_{\gamma}u_{\bar{\rho}\alpha}-u_{\bar{\beta}\alpha}\Gamma_{\gamma\bar{\rho}}^{\bar{\beta}}-u_{\bar{\rho}\beta}\Gamma_{\gamma\alpha}^{\beta}+\frac{i}{2}Q_{\alpha\gamma}^{\bar{\mu}}u_{\bar{\rho}\bar{\mu}}\bigg)\\
\notag &\ \ \ -\frac{i}{2}Q_{\alpha\gamma}^{\bar{\mu}}u_{\bar{\rho}\bar{\mu}}-2ih_{\gamma\bar{\rho}}u_{0\alpha}+u_{\beta}R_{\alpha\ \gamma\bar{\rho}}^{\ \beta}+\frac{i}{2}Q_{\alpha\gamma}^{\bar{\mu}}\bigg(W_{\bar{\rho}}u_{\bar{\mu}}-\Gamma_{\bar{\rho}\bar{\mu}}^{\beta}u_{\beta}-\Gamma_{\bar{\rho}\bar{\mu}}^{\bar{\beta}}u_{\bar{\beta}}\bigg)\\
\notag &\ \ \ +\frac{i}{2}\bigg(W_{\bar{\rho}}Q_{\alpha\gamma}^{\bar{\mu}}-\Gamma_{\bar{\rho}\alpha}^{\beta}Q_{\beta\gamma}^{\bar{\mu}}
-\Gamma_{\bar{\rho}\gamma}^{\beta}Q_{\alpha\beta}^{\bar{\mu}}+\Gamma_{\bar{\rho}\bar{\beta}}^{\bar{\mu}}Q_{\alpha\gamma}^{\bar{\beta}}\bigg)u_{\bar{\mu}}\\
\notag &=-u_{\bar{\rho}\gamma\alpha}+u_{\gamma\bar{\rho}\alpha}-2ih_{\gamma\bar{\rho}}u_{0\alpha}+u_{\beta}R_{\alpha\ \gamma\bar{\rho}}^{\ \beta}+\frac{i}{2}Q_{\alpha\gamma,\bar{\rho}}^{\bar{\mu}}u_{\bar{\mu}},
\end{align}
by \eqref{eq2.86}, \eqref{SD} and \eqref{TD}. The first identity of \eqref{OCF} holds.

To prove the second identity in \eqref{OCF}, we consider the components of $\theta^{\bar{\gamma}}\wedge\theta^{\bar{\rho}}$ in \eqref{eq3.18} to get
\begin{align}
\notag 0&=\bigg(W_{\bar{\gamma}}u_{\bar{\rho}\alpha}+u_{\beta\alpha}\Gamma_{\bar{\rho}\bar{\gamma}}^{\beta}
+u_{\bar{\beta}\alpha}\Gamma_{\bar{\rho}\bar{\gamma}}^{\bar{\beta}}+u_{\beta}\Gamma_{\bar{\gamma}\alpha}^{\mu}\Gamma_{\bar{\rho}\mu}^{\beta}
+\frac{1}{2}u_{\beta}R_{\alpha\ \bar{\gamma}\bar{\rho}}^{\ \beta}+W_{\bar{\gamma}}u_{\beta}\Gamma_{\bar{\rho}\alpha}^{\beta}
-\frac{i}{2}u_{\bar{\mu}}Q_{\alpha\beta}^{\bar{\mu}}\Gamma_{\bar{\rho}\bar{\gamma}}^{\beta}\bigg)\theta^{\bar{\gamma}}\wedge\theta^{\bar{\rho}}\\
\notag &=\bigg(W_{\bar{\gamma}}u_{\bar{\rho}\alpha}-u_{\beta\alpha}\Gamma_{\bar{\gamma}\bar{\rho}}^{\beta}
-u_{\bar{\beta}\alpha}\Gamma_{\bar{\gamma}\bar{\rho}}^{\bar{\beta}}-u_{\bar{\rho}\beta}\Gamma_{\bar{\gamma}\alpha}^{\beta}
+\frac{1}{2}u_{\beta}R_{\alpha\ \bar{\gamma}\bar{\rho}}^{\ \beta}-\frac{i}{2}u_{\bar{\mu}}Q_{\alpha\beta}^{\bar{\mu}}\Gamma_{\bar{\rho}\bar{\gamma}}^{\beta}\bigg)\theta^{\bar{\gamma}}\wedge\theta^{\bar{\rho}}\\
\notag &=\bigg(u_{\bar{\gamma}\bar{\rho}\alpha}+\frac{1}{2}u_{\beta}R_{\alpha\ \bar{\gamma}\bar{\rho}}^{\ \beta}+\frac{1}{4}u_{\bar{\mu}}Q_{\alpha\beta}^{\bar{\mu}}Q_{\bar{\gamma}\bar{\rho}}^{\beta}\bigg)\theta^{\bar{\gamma}}\wedge\theta^{\bar{\rho}},
\end{align}
by
$$\bigg(W_{\bar{\gamma}}u_{\beta}\Gamma_{\bar{\rho}\alpha}^{\beta}+u_{\beta}\Gamma_{\bar{\gamma}\alpha}^{\mu}\Gamma_{\bar{\rho}\mu}^{\beta}\bigg)\theta^{\bar{\gamma}}\wedge\theta^{\bar{\rho}}
=\bigg(-W_{\bar{\rho}}u_{\beta}\Gamma_{\bar{\gamma}\alpha}^{\beta}+u_{\mu}\Gamma_{\bar{\rho}\beta}^{\mu}\Gamma_{\bar{\gamma}\alpha}^{\beta}\bigg)\theta^{\bar{\gamma}}\wedge\theta^{\bar{\rho}}
=-u_{\bar{\rho}\beta}\Gamma_{\bar{\gamma}\alpha}^{\beta}\theta^{\bar{\gamma}}\wedge\theta^{\bar{\rho}},$$
by using \eqref{SD}. Equivalently, we have
\begin{align}
\notag 0&=u_{\bar{\gamma}\bar{\rho}\alpha}-u_{\bar{\rho}\bar{\gamma}\alpha}+\frac{1}{2}u_{\beta}R_{\alpha\ \bar{\gamma}\bar{\rho}}^{\ \beta}-\frac{1}{2}u_{\beta}R_{\alpha\ \bar{\rho}\bar{\gamma}}^{\ \beta}+\frac{1}{4}u_{\bar{\mu}}Q_{\alpha\beta}^{\bar{\mu}}\big(Q_{\bar{\gamma}\bar{\rho}}^{\beta}-Q_{\bar{\rho}\bar{\gamma}}^{\beta}\big)\\
\notag &=u_{\bar{\gamma}\bar{\rho}\alpha}-u_{\bar{\rho}\bar{\gamma}\alpha}+u_{\beta}R_{\alpha\ \bar{\gamma}\bar{\rho}}^{\ \beta}+\frac{1}{4}Q_{\alpha\beta}^{\bar{\mu}}Q_{\bar{\gamma}\bar{\beta}}^{\rho}u_{\bar{\mu}}=u_{\bar{\gamma}\bar{\rho}\alpha}-u_{\bar{\rho}\bar{\gamma}\alpha}+u_{\beta}R_{\alpha\ \bar{\gamma}\bar{\rho}}^{\ \beta}-\frac{1}{4}Q_{\alpha\beta}^{\bar{\mu}}Q_{\bar{\rho}\bar{\beta}}^{\gamma}u_{\bar{\mu}},
\end{align}
by \eqref{eq2.9}. This together with $R_{\alpha\ \bar{\gamma}\bar{\rho}}^{\ \beta}=2i(A_{\bar{\gamma}}^{\beta}h_{\alpha\bar{\rho}}-A_{\bar{\rho}}^{\beta}h_{\alpha\bar{\gamma}})-\frac{i}{2}h^{\beta\bar{\sigma}}h_{\lambda\bar{\gamma}}Q_{\bar{\rho}\bar{\sigma},\alpha}^{\lambda}$
(cf. (43) in \cite{BD1}) implies the second identity in \eqref{OCF}.
\end{proof}
\begin{rem}
Note that when $Q\equiv0$, the commutative formulae \eqref{OCF} is the same as Proposition 9.2 in \cite{BD1} in pseudohermitian case.
\end{rem}

\section{The Bochner-type formula}
By definition, we have $(\nabla_{H}u)^{\alpha}=h(\nabla_{H}u,W_{\bar{\alpha}})=W_{\bar{\alpha}}u=u_{\bar{\alpha}}$ for any $u{\in}C_0^{\infty}(M)$. Thus $\nabla_{H}u=(\nabla_{H}u)^{\alpha}W_{\alpha}+(\nabla_{H}u)^{\bar{\alpha}}W_{\bar{\alpha}}
=u_{\bar{\alpha}}W_{\alpha}+u_{\alpha}W_{\bar{\alpha}}$. Consequently, we have $\partial_{b}u=u_{\bar{\alpha}}W_{\alpha}$ and $\|\partial_{b}u\|^2=\|u_{\bar{\lambda}}W_{\lambda}\|^2=u_{\lambda}u_{\bar{\lambda}}$.
\begin{thm}\label{thm1.1}
Under an orthonormal $T^{(1,0)}M$-frame, the Bochner-type formula holds in the following form:
\begin{align}\label{eq1.1}
\notag \Delta_{b}(\|\partial_{b}u\|^2)&=2\big(u_{\alpha\lambda}u_{\bar{\alpha}\bar{\lambda}}+u_{\alpha\bar{\lambda}}u_{\bar{\alpha}\lambda}\big)+4i(u_{\alpha}u_{0\bar{\alpha}}-u_{\bar{\alpha}}u_{0\alpha})\\
\notag &\ \ \ +2ni(A_{\bar{\alpha}\bar{\beta}}u_{\alpha}u_{\beta}-A_{\alpha\beta}u_{\bar{\alpha}}u_{\bar{\beta}})+2R_{\alpha\bar{\beta}}u_{\bar{\alpha}}u_{\beta}
+u_{\alpha}(\Delta_bu)_{\bar{\alpha}}+u_{\bar{\alpha}}(\Delta_bu)_{\alpha}\\
&\ \ \ +i\bigg(Q_{\bar{\alpha}\bar{\beta},\gamma}^{\gamma}u_{\alpha}u_{\beta}-Q_{\alpha\beta,\bar{\gamma}}^{\bar{\gamma}}u_{\bar{\alpha}}u_{\bar{\beta}}\bigg)
-\frac{1}{2}Q_{\alpha\gamma}^{\bar{\rho}}Q_{\bar{\beta}\bar{\gamma}}^{\rho}u_{\bar{\alpha}}u_{\beta}.
\end{align}
\end{thm}
To prove it, we need a lemma.
\begin{lem}\label{lem4.1}
For any $u{\in}C_0^{\infty}(M)$,
\begin{equation}\label{eq4.2}
\Delta_{b}(\|\partial_{b}u\|^2)=S_1+S_2,
\end{equation}
where
$S_1=2\big(u_{\alpha\lambda}u_{\bar{\alpha}\bar{\lambda}}+u_{\alpha\bar{\lambda}}u_{\bar{\alpha}\lambda}\big)
$, $S_2=u_{\bar{\lambda}}u_{\alpha\bar{\alpha}\lambda}+u_{\lambda}u_{\alpha\bar{\alpha}\bar{\lambda}}
+u_{\lambda}u_{\bar{\alpha}\alpha\bar{\lambda}}+u_{\bar{\lambda}}u_{\bar{\alpha}\alpha\lambda}.$
\end{lem}
\begin{proof}
We claim that
\begin{equation}\label{eq4.1} \big(\|\partial_{b}u\|^2\big)_{\alpha\bar{\alpha}}=u_{\alpha\lambda}u_{\bar{\alpha}\bar{\lambda}}+u_{\alpha\bar{\lambda}}u_{\bar{\alpha}\lambda}
+u_{\bar{\lambda}}u_{\alpha\bar{\alpha}\lambda}+u_{\lambda}u_{\alpha\bar{\alpha}\bar{\lambda}}.
\end{equation}
Then \eqref{eq4.2} follows directly by taking summation of \eqref{eq4.1} and its conjugation.

By \eqref{eq2.9}, we have $$\Gamma_{\alpha\lambda}^{\beta}u_{\beta}u_{\bar{\lambda}}+\Gamma_{\alpha\bar{\lambda}}^{\bar{\beta}}u_{\lambda}u_{\bar{\beta}}
=-\Gamma_{\alpha\bar{\beta}}^{\bar{\lambda}}u_{\beta}u_{\bar{\lambda}}+\Gamma_{\alpha\bar{\lambda}}^{\bar{\beta}}u_{\lambda}u_{\bar{\beta}}=0,$$
and $\Gamma_{\alpha\lambda}^{\bar{\beta}}u_{\bar{\beta}}u_{\bar{\lambda}}=-\Gamma_{\alpha\beta}^{\bar{\lambda}}u_{\bar{\beta}}u_{\bar{\lambda}}
=-\Gamma_{\alpha\lambda}^{\bar{\beta}}u_{\bar{\beta}}u_{\bar{\lambda}}$.
Thus, $\Gamma_{\alpha\lambda}^{\bar{\beta}}u_{\bar{\beta}}u_{\bar{\lambda}}=0.$
So by \eqref{SD}, we get
\begin{align}\label{eq4.1a}
  \notag W_{\alpha}\big(\|\partial_{b}u\|^2\big)&=W_{\alpha}(u_{\lambda})u_{\bar{\lambda}}+u_{\lambda}W_{\alpha}(u_{\bar{\lambda}})\\
  &=\bigg(u_{\alpha\lambda}+\Gamma_{\alpha\lambda}^{\beta}u_{\beta}+\Gamma_{\alpha\lambda}^{\bar{\beta}}u_{\bar{\beta}}\bigg)u_{\bar{\lambda}}
  +u_{\lambda}\bigg(u_{\alpha\bar{\lambda}}+\Gamma_{\alpha\bar{\lambda}}^{\bar{\beta}}u_{\bar{\beta}}\bigg)
  =u_{\alpha\lambda}u_{\bar{\lambda}}+u_{\lambda}u_{\alpha\bar{\lambda}}.
\end{align}
Then taking conjugation of \eqref{eq4.1a}, we get
$W_{\bar{\alpha}}\big(\|\partial_{b}u\|^2\big)=u_{\bar{\alpha}\bar{\lambda}}u_{\lambda}+u_{\bar{\lambda}}u_{\bar{\alpha}\lambda}.$
So
\begin{align}
  \notag \big(\|\partial_{b}u\|^2\big)_{\alpha\bar{\alpha}}&=W_{\alpha}W_{\bar{\alpha}}(\|\partial_{b}u\|^2)-\Gamma_{\alpha\bar{\alpha}}^{\bar{\gamma}}W_{\bar{\gamma}}(\|\partial_{b}u\|^2)\\
  \notag &=W_{\alpha}\big(u_{\bar{\alpha}\bar{\lambda}}u_{\lambda}+u_{\bar{\lambda}}u_{\bar{\alpha}\lambda}\big)-\Gamma_{\alpha\bar{\alpha}}^{\bar{\gamma}}\big(u_{\bar{\gamma}\bar{\lambda}}u_{\lambda}+u_{\bar{\gamma}\lambda}u_{\bar{\lambda}}\big)\\
  \notag &=W_{\alpha}(u_{\lambda})u_{\bar{\alpha}\bar{\lambda}}+W_{\alpha}(u_{\bar{\lambda}})u_{\bar{\alpha}\lambda}
  +\bigg(W_{\alpha}(u_{\bar{\alpha}\bar{\lambda}})-\Gamma_{\alpha\bar{\alpha}}^{\bar{\gamma}}u_{\bar{\gamma}\bar{\lambda}}\bigg)u_{\lambda}
  +\bigg(W_{\alpha}(u_{\bar{\alpha}\lambda})-\Gamma_{\alpha\bar{\alpha}}^{\bar{\gamma}}u_{\bar{\gamma}\lambda}\bigg)u_{\bar{\lambda}}\\
  \notag &=\bigg(u_{\alpha\lambda}+\Gamma_{\alpha\lambda}^{\gamma}u_{\gamma}+\Gamma_{\alpha\lambda}^{\bar{\gamma}}u_{\bar{\gamma}}\bigg)u_{\bar{\alpha}\bar{\lambda}}
  +\bigg(u_{\alpha\bar{\lambda}}+\Gamma_{\alpha\bar{\lambda}}^{\bar{\gamma}}u_{\bar{\gamma}}\bigg)u_{\bar{\alpha}\lambda}\\
  \notag &\ \ \ +\bigg(u_{\alpha\bar{\alpha}\bar{\lambda}}+\Gamma_{\alpha\bar{\lambda}}^{\bar{\gamma}}u_{\bar{\alpha}\bar{\gamma}}\bigg)u_{\lambda}
  +\bigg(u_{\alpha\bar{\alpha}\lambda}+\Gamma_{\alpha\lambda}^{\gamma}u_{\bar{\alpha}\gamma}+\Gamma_{\alpha\lambda}^{\bar{\gamma}}u_{\bar{\alpha}\bar{\gamma}}\bigg)u_{\bar{\lambda}}\\
  \notag &=u_{\alpha\lambda}u_{\bar{\alpha}\bar{\lambda}}+u_{\alpha\bar{\lambda}}u_{\bar{\alpha}\lambda}
  +u_{\lambda}u_{\alpha\bar{\alpha}\bar{\lambda}}+u_{\bar{\lambda}}u_{\alpha\bar{\alpha}\lambda}\\
  \notag &\ \ \ +\bigg(\Gamma_{\alpha\lambda}^{\gamma}u_{\gamma}u_{\bar{\alpha}\bar{\lambda}}+\Gamma_{\alpha\bar{\lambda}}^{\bar{\gamma}}u_{\lambda}u_{\bar{\alpha}\bar{\gamma}}\bigg)
  +\bigg(\Gamma_{\alpha\lambda}^{\bar{\gamma}}u_{\bar{\gamma}}u_{\bar{\alpha}\bar{\lambda}}+\Gamma_{\alpha\lambda}^{\bar{\gamma}}u_{\bar{\lambda}}u_{\bar{\alpha}\bar{\gamma}}\bigg)
  +\bigg(\Gamma_{\alpha\bar{\lambda}}^{\bar{\gamma}}u_{\bar{\gamma}}u_{\bar{\alpha}\lambda}+\Gamma_{\alpha\lambda}^{\gamma}u_{\bar{\lambda}}u_{\bar{\alpha}\gamma}\bigg),
\end{align}
where we have used \eqref{SD} and \eqref{TD}. The result follows from
$$\Gamma_{\alpha\lambda}^{\gamma}u_{\gamma}u_{\bar{\alpha}\bar{\lambda}}+\Gamma_{\alpha\bar{\lambda}}^{\bar{\gamma}}u_{\lambda}u_{\bar{\alpha}\bar{\gamma}}
=-\Gamma_{\alpha\bar{\gamma}}^{\bar{\lambda}}u_{\gamma}u_{\bar{\alpha}\bar{\lambda}}+\Gamma_{\alpha\bar{\lambda}}^{\bar{\gamma}}u_{\lambda}u_{\bar{\alpha}\bar{\gamma}}=0,$$
and similarly,
$\Gamma_{\alpha\lambda}^{\bar{\gamma}}u_{\bar{\gamma}}u_{\bar{\alpha}\bar{\lambda}}+\Gamma_{\alpha\lambda}^{\bar{\gamma}}u_{\bar{\lambda}}u_{\bar{\alpha}\bar{\gamma}}=0,$
$\Gamma_{\alpha\bar{\lambda}}^{\bar{\gamma}}u_{\bar{\gamma}}u_{\bar{\alpha}\lambda}+\Gamma_{\alpha\lambda}^{\gamma}u_{\bar{\lambda}}u_{\bar{\alpha}\gamma}=0$ by \eqref{eq2.9}.
\end{proof}
\eqref{eq4.2} coincides with (9.34) in \cite{DT1} in the pseudohermitian case.
{\vskip 5mm}
{\it Proof of theorem\ref{thm1.1}.}
Note that by \eqref{eq6.1a}, we have $(\Delta_{b}u)_c=u_{c\alpha\bar{\alpha}}+u_{c\bar{\alpha}\alpha}$. We hope to express third-order covariant derivatives in \eqref{eq4.2} in terms of $(\Delta_{b}u)_c$. To do so, we apply inner and outer commutation formulae to \eqref{eq4.2} to express $u_{\alpha\bar{\alpha}b}$ and $u_{\bar{\alpha}{\alpha}b}$ in terms of $u_{b\alpha\bar{\alpha}}$ and $u_{b\bar{\alpha}{\alpha}}$, respectively. By \eqref{ICF}, we have the following inner commutation formulae.
\begin{equation}
\begin{aligned}
\notag &u_{\alpha\bar{\alpha}\lambda}=u_{\alpha\lambda\bar{\alpha}}-2ih_{\lambda\bar{\alpha}}u_{\alpha0},{\quad}
u_{\alpha\bar{\alpha}\bar{\lambda}}=u_{\alpha\bar{\lambda}\bar{\alpha}},\\
\notag &u_{\bar{\alpha}\alpha\bar{\lambda}}=u_{\bar{\alpha}\bar{\lambda}\alpha}+2ih_{\alpha\bar{\lambda}}u_{\bar{\alpha}0},{\quad}
\notag u_{\bar{\alpha}\alpha\lambda}=u_{\bar{\alpha}\lambda\alpha}.
\end{aligned}
\end{equation}
So $S_2$ in \eqref{eq4.2} becomes
\begin{align}\label{eq4.3}
\notag S_2&=u_{\bar{\lambda}}(u_{\alpha\lambda\bar{\alpha}}-2ih_{\lambda\bar{\alpha}}u_{\alpha0})
+u_{\lambda}u_{\alpha\bar{\lambda}\bar{\alpha}}+u_{\lambda}(u_{\bar{\alpha}\bar{\lambda}\alpha}
+2ih_{\alpha\bar{\lambda}}u_{\bar{\alpha}0})+u_{\bar{\lambda}}u_{\bar{\alpha}\lambda\alpha}\\
&=u_{\bar{\lambda}}u_{\alpha\lambda\bar{\alpha}}+u_{\lambda}u_{\alpha\bar{\lambda}\bar{\alpha}}
+u_{\lambda}u_{\bar{\alpha}\bar{\lambda}\alpha}+u_{\bar{\lambda}}u_{\bar{\alpha}\lambda\alpha}-2iu_{\bar{\alpha}}u_{0\alpha}+2iu_{\alpha}u_{0\bar{\alpha}}
-2iA_{\alpha}^{\bar{\beta}}u_{\bar{\alpha}}u_{\bar{\beta}}+2iA_{\bar{\alpha}}^{\beta}u_{\alpha}u_{\beta},
\end{align}
by $u_{\alpha0}=u_{0\alpha}+A_{\alpha}^{\bar{\beta}}u_{\bar{\beta}}$, $u_{\bar{\alpha}0}=u_{0\bar{\alpha}}+A_{\bar{\alpha}}^{\beta}u_{\beta}$ in \eqref{SDCF}.
The outer commutation formulae \eqref{OCF} for $\rho=\alpha$ can be written as
\begin{equation}\label{eq4.6}
\begin{aligned}
&u_{\bar{\alpha}\lambda\alpha}=u_{\lambda\bar{\alpha}\alpha}-2iu_{0\lambda}+R_{\lambda\bar{\beta}}u_{\beta}
-\frac{i}{2}Q_{\beta\lambda,\bar{\alpha}}^{\bar{\alpha}}u_{\bar{\beta}},\\
&u_{\bar{\alpha}\bar{\lambda}\alpha}=u_{\bar{\lambda}\bar{\alpha}\alpha}+2i(n-1)A_{\bar{\lambda}}^{\beta}u_{\beta}
+\frac{i}{2}Q_{\bar{\lambda}\bar{\beta},\alpha}^{\alpha}u_{\beta}-\frac{1}{4}Q_{\alpha\beta}^{\bar{\mu}}Q_{\bar{\alpha}\bar{\beta}}^{\lambda}u_{\bar{\mu}}.
\end{aligned}
\end{equation}
by $Q_{\alpha\lambda,\bar{\alpha}}^{\bar{\beta}}=-Q_{\beta\lambda,\bar{\alpha}}^{\bar{\alpha}}$ in \eqref{eq2.9} and
\begin{equation}\label{eq4.7}
R_{\alpha\ \lambda\bar{\alpha}}^{\ \beta}=h^{\beta\bar{\mu}}R_{\alpha\bar{\mu}\lambda\bar{\alpha}}=R_{\alpha\bar{\beta}\lambda\bar{\alpha}}
=R_{\bar{\beta}\alpha\bar{\alpha}\lambda}=R_{\bar{\alpha}\alpha\bar{\beta}\lambda}=R_{\alpha\bar{\alpha}\lambda\bar{\beta}}=R_{\lambda\bar{\alpha}\alpha\bar{\beta}}
=R_{\lambda\ \alpha\bar{\beta}}^{\ \alpha}=R_{\lambda\bar{\beta}},
\end{equation}
by using Proposition \ref{pro2.5} repeatedly. Taking conjugation of \eqref{eq4.6} and noting that
$R_{\bar{\lambda}\beta}=R_{\bar{\lambda}\ \bar{\alpha}\beta}^{\ \bar{\alpha}}=R_{\bar{\lambda}\alpha\bar{\alpha}\beta}=R_{\beta\bar{\lambda}}$ by \eqref{eq4.7},
we get
\begin{equation}\label{eq4.8}
\begin{aligned}
&u_{\alpha\bar{\lambda}\bar{\alpha}}=u_{\bar{\lambda}\alpha\bar{\alpha}}+2iu_{0\bar{\lambda}}
+R_{\beta\bar{\lambda}}u_{\bar{\beta}}+\frac{i}{2}Q_{\bar{\beta}\bar{\lambda},\alpha}^{\alpha}u_{\beta},\\ &u_{\alpha\lambda\bar{\alpha}}=u_{\lambda\alpha\bar{\alpha}}-2i(n-1)A_{\lambda}^{\bar{\beta}}u_{\bar{\beta}}-\frac{i}{2}Q_{\lambda\beta,\bar{\alpha}}^{\bar{\alpha}}u_{\bar{\beta}}
-\frac{1}{4}Q_{\alpha\beta}^{\bar{\lambda}}Q_{\bar{\alpha}\bar{\beta}}^{\mu}u_{\mu}.
\end{aligned}
\end{equation}
Substitute \eqref{eq4.6} and \eqref{eq4.8} to \eqref{eq4.3} to get
\begin{align}\label{eq4.10}
\notag S_2&=u_{\bar{\lambda}}u_{\lambda\alpha\bar{\alpha}}-2i(n-1)A_{\lambda}^{\bar{\beta}}u_{\bar{\lambda}}u_{\bar{\beta}}
-\frac{i}{2}Q_{\lambda\beta,\bar{\alpha}}^{\bar{\alpha}}u_{\bar{\lambda}}u_{\bar{\beta}}
-\frac{1}{4}Q_{\alpha\beta}^{\bar{\lambda}}Q_{\bar{\alpha}\bar{\beta}}^{\mu}u_{\mu}u_{\bar{\lambda}}\\
\notag &\ \ \ +u_{\lambda}u_{\bar{\lambda}\alpha\bar{\alpha}}+2iu_{\lambda}u_{0\bar{\lambda}}
+R_{\beta\bar{\lambda}}u_{\lambda}u_{\bar{\beta}}+\frac{i}{2}Q_{\bar{\beta}\bar{\lambda},\alpha}^{\alpha}u_{\lambda}u_{\beta}\\
\notag &\ \ \ +u_{\lambda}u_{\bar{\lambda}\bar{\alpha}\alpha}+2i(n-1)A_{\bar{\lambda}}^{\beta}u_{\lambda}u_{\beta}
+\frac{i}{2}Q_{\bar{\lambda}\bar{\beta},\alpha}^{\alpha}u_{\lambda}u_{\beta}-\frac{1}{4}Q_{\alpha\beta}^{\bar{\mu}}Q_{\bar{\alpha}\bar{\beta}}^{\lambda}u_{\lambda}u_{\bar{\mu}}\\
\notag &\ \ \ +u_{\bar{\lambda}}u_{\lambda\bar{\alpha}\alpha}-2iu_{\bar{\lambda}}u_{0\lambda}+R_{\lambda\bar{\beta}}u_{\beta}u_{\bar{\lambda}}
-\frac{i}{2}Q_{\beta\lambda,\bar{\alpha}}^{\bar{\alpha}}u_{\bar{\beta}}u_{\bar{\lambda}}\\
\notag &\ \ \ -2iu_{\bar{\alpha}}u_{0\alpha}+2iu_{\alpha}u_{0\bar{\alpha}}-2iA_{\alpha}^{\bar{\beta}}u_{\bar{\alpha}}u_{\bar{\beta}}+2iA_{\bar{\alpha}}^{\beta}u_{\alpha}u_{\beta}\\
\notag &=u_{\lambda}(\Delta_{b}u)_{\bar{\lambda}}+u_{\bar{\lambda}}(\Delta_{b}u)_{\lambda}+4i(u_{\alpha}u_{0\bar{\alpha}}-u_{\bar{\alpha}}u_{0\alpha})+2ni(A_{\bar{\alpha}\bar{\beta}}u_{\alpha}u_{\beta}-A_{\alpha\beta}u_{\bar{\alpha}}u_{\bar{\beta}})
+2R_{\alpha\bar{\beta}}u_{\bar{\alpha}}u_{\beta}\\
&\ \ \ +iQ_{\bar{\lambda}\bar{\beta},\alpha}^{\alpha}u_{\lambda}u_{\beta}-iQ_{\lambda\beta,\bar{\alpha}}^{\bar{\alpha}}u_{\bar{\lambda}}u_{\bar{\beta}}
-\frac{1}{2}Q_{\mu\beta}^{\bar{\alpha}}Q_{\bar{\lambda}\bar{\beta}}^{\alpha}u_{\lambda}u_{\bar{\mu}}.
\end{align}
The last identity follows from $u_{\lambda}u_{\bar{\lambda}\alpha\bar{\alpha}}+u_{\lambda}u_{\bar{\lambda}\bar{\alpha}\alpha}=u_{\lambda}(\Delta_{b}u)_{\bar{\lambda}}$ by \eqref{eq6.1a}, $A_{\alpha}^{\bar{\beta}}=h^{\gamma\bar{\beta}}A_{\alpha\gamma}=A_{\alpha\beta}$ and $Q_{\beta\alpha}^{\bar{\gamma}}=-Q_{\gamma\alpha}^{\bar{\beta}}$ in \eqref{eq2.9}.
Substituting \eqref{eq4.10} to \eqref{eq4.2}, we get \eqref{eq1.1}.
\begin{rem}\label{rem4.2}
When $Q\equiv0$, the Bochner-type formula \eqref{eq1.1} is the same as the pseudohermitian case (see (9.36) in \cite{DT1} or Theorem 6 in \cite{LW1}).
\end{rem}
\section{Two useful identities}
We need the following lemma to handle the second bracket in the Bochner type formula \eqref{eq1.1}.
\begin{lem}
For any $u\in{C_0^{\infty}(M)}$, we have
\begin{align}\label{eq5.1}
\notag \int_Mi(u_{0\bar{\alpha}}u_{\alpha}-u_{0\alpha}u_{\bar{\alpha}})dV
=&\frac{1}{n}\int_M\bigg(u_{\bar{\alpha}\beta}u_{\alpha\bar{\beta}}-u_{\alpha\beta}u_{\bar{\alpha}\bar{\beta}}
-R_{\alpha\bar{\beta}}u_{\bar{\alpha}}u_{\beta}-\frac{i}{2}Q_{\bar{\alpha}\bar{\beta},\gamma}^{\gamma}u_{\alpha}u_{\beta}\\
&+\frac{i}{2}Q_{\alpha\beta,\bar{\gamma}}^{\bar{\gamma}}u_{\bar{\alpha}}u_{\bar{\beta}}
-\frac{1}{2}Q_{\alpha\gamma}^{\bar{\rho}}Q_{\bar{\beta}\bar{\gamma}}^{\rho}u_{\bar{\alpha}}u_{\beta}
+Q_{\alpha\gamma}^{\bar{\rho}}Q_{\bar{\beta}\bar{\rho}}^{\gamma}u_{\bar{\alpha}}u_{\beta}\bigg)dV,
\end{align}
and
\begin{equation}\label{eq5.2}
\int_Mi(u_{0\bar{\alpha}}u_{\alpha}-u_{0\alpha}u_{\bar{\alpha}})dV=\int_M\bigg(-\frac{2}{n}\bigg|\sum\limits_{\alpha}u_{\alpha\bar{\alpha}}\bigg|^2+\frac{1}{2n}(\Delta_bu)^2
+iA_{\alpha\beta}u_{\bar{\alpha}}u_{\bar{\beta}}-iA_{\bar{\alpha}\bar{\beta}}u_{\alpha}u_{\beta}\bigg)dV.
\end{equation}
\end{lem}
\begin{rem}
\eqref{eq5.1} and \eqref{eq5.2} are the same as the corresponding identities in the pseudohermitian case when $Q\equiv0$ (cf. Lemma 9.1 in \cite{DT1} or Lemma 4 and Lemma 5 in \cite{G1}).
\end{rem}
\subsection{The Proof of \eqref{eq5.1}}
By definition we have
\begin{align}\label{eq5.3}
\notag \int_Mu_{\alpha\beta}u_{\bar{\alpha}\bar{\beta}}&dV=\sum\limits_{\alpha,\beta}(u_{\alpha\beta},u_{\alpha\beta})
=\sum\limits_{\alpha,\beta,\gamma}\bigg(u_{\alpha\beta},W_{\alpha}(u_{\beta})-\Gamma_{\alpha\beta}^{\gamma}u_{\gamma}-\Gamma_{\alpha\beta}^{\bar{\gamma}}u_{\bar{\gamma}}\bigg)\\
\notag &=\sum\limits_{\alpha,\beta}\bigg(-W_{\bar{\alpha}}(u_{\alpha\beta})+\Gamma_{\bar{\gamma}\gamma}^{\alpha}u_{\alpha\beta}
-\Gamma_{\bar{\alpha}\bar{\gamma}}^{\bar{\beta}}u_{\alpha\gamma},u_{\beta}\bigg)
-\bigg(\Gamma_{\bar{\alpha}\bar{\gamma}}^{\beta}u_{\alpha\gamma},u_{\bar{\beta}}\bigg)\\ &=\sum\limits_{\alpha,\beta,\gamma}\Bigg(-W_{\bar{\alpha}}\bigg(W_{\alpha}(u_{\beta})-\Gamma_{\alpha\beta}^{\gamma}u_{\gamma}
-\Gamma_{\alpha\beta}^{\bar{\gamma}}u_{\bar{\gamma}}\bigg)
+\Gamma_{\bar{\gamma}\gamma}^{\alpha}u_{\alpha\beta}
-\Gamma_{\bar{\alpha}\bar{\gamma}}^{\bar{\beta}}u_{\alpha\gamma},u_{\beta}\Bigg)
-\bigg(\Gamma_{\bar{\alpha}\bar{\gamma}}^{\beta}u_{\alpha\gamma},u_{\bar{\beta}}\bigg)
\end{align}
by \eqref{SD} and Lemma \ref{lem3.1}. Apply
$[W_{\bar{\alpha}},W_{\alpha}]=\nabla_{W_{\bar{\alpha}}}W_{\alpha}-\nabla_{W_{\alpha}}W_{\bar{\alpha}}-\tau(W_{\bar{\alpha}},W_{\alpha})
=\Gamma_{\bar{\alpha}\alpha}^{\gamma}W_{\gamma}-\Gamma_{\alpha\bar{\alpha}}^{\bar{\gamma}}W_{\bar{\gamma}}-2h(W_{\bar{\alpha}},JW_{\alpha})T
=\Gamma_{\bar{\alpha}\alpha}^{\gamma}W_{\gamma}-\Gamma_{\alpha\bar{\alpha}}^{\bar{\gamma}}W_{\bar{\gamma}}-2i\delta_{\alpha\bar{\alpha}}T$
to \eqref{eq5.3} to get that $\int_Mu_{\alpha\beta}u_{\bar{\alpha}\bar{\beta}}dV$ equals to
\begin{align}\label{eq5.4}
&\ \ \ \notag -\sum\limits_{\alpha,\beta}\big(W_{\alpha}W_{\bar{\alpha}}(u_{\beta}),u_{\beta}\big)+2ni\sum\limits_{\beta}\big(T(u_{\beta}),u_{\beta}\big)
+\sum\limits_{\alpha,\beta,\gamma}\Bigg(-\Gamma_{\bar{\alpha}\alpha}^{\gamma}W_{\gamma}(u_{\beta})+\Gamma_{\alpha\bar{\alpha}}^{\bar{\gamma}}W_{\bar{\gamma}}(u_{\beta})\\
\notag &\ \ \ +W_{\bar{\alpha}}\bigg(\Gamma_{\alpha\beta}^{\gamma}u_{\gamma}+\Gamma_{\alpha\beta}^{\bar{\gamma}}u_{\bar{\gamma}}\bigg) +\Gamma_{\bar{\gamma}\gamma}^{\alpha}u_{\alpha\beta}-\Gamma_{\bar{\alpha}\bar{\gamma}}^{\bar{\beta}}u_{\alpha\gamma},u_{\beta}\Bigg)
-\sum\limits_{\alpha,\beta,\gamma}\big(\Gamma_{\bar{\alpha}\bar{\gamma}}^{\beta}u_{\alpha\gamma},u_{\bar{\beta}}\big)\\
\notag &=-\sum\limits_{\alpha,\beta}\bigg(W_{\alpha}(u_{\bar{\alpha}\beta}+\Gamma_{\bar{\alpha}\beta}^{\gamma}u_{\gamma}),u_{\beta}\bigg)
+2ni\sum\limits_{\beta}\big(T(u_{\beta}),u_{\beta}\big)\\
\notag &\ \ \ +\sum\limits_{\alpha,\beta,\gamma}\bigg(-\Gamma_{\bar{\alpha}\alpha}^{\gamma}W_{\gamma}(u_{\beta})+\Gamma_{\alpha\bar{\alpha}}^{\bar{\gamma}}W_{\bar{\gamma}}(u_{\beta})
+W_{\bar{\alpha}}(\Gamma_{\alpha\beta}^{\gamma})u_{\gamma}+\Gamma_{\alpha\beta}^{\gamma}W_{\bar{\alpha}}(u_{\gamma})
+W_{\bar{\alpha}}(\Gamma_{\alpha\beta}^{\bar{\gamma}})u_{\bar{\gamma}}+\Gamma_{\alpha\beta}^{\bar{\gamma}}W_{\bar{\alpha}}(u_{\bar{\gamma}})\\
\notag &\ \ \ +\Gamma_{\bar{\gamma}\gamma}^{\alpha}u_{\alpha\beta}-\Gamma_{\bar{\alpha}\bar{\gamma}}^{\bar{\beta}}u_{\alpha\gamma},u_{\beta}\bigg)
-\sum\limits_{\alpha,\beta,\gamma}\big(\Gamma_{\bar{\alpha}\bar{\gamma}}^{\beta}u_{\alpha\gamma},u_{\bar{\beta}}\big)\\
\notag &=-\sum\limits_{\alpha,\beta}\big(W_{\alpha}(u_{\bar{\alpha}\beta}),u_{\beta}\big)
-\sum\limits_{\alpha,\beta,\gamma}\big(W_{\alpha}(\Gamma_{\bar{\alpha}\beta}^{\gamma})u_{\gamma},u_{\beta}\big)
-\sum\limits_{\alpha,\beta,\gamma,\rho}\big(\Gamma_{\bar{\alpha}\beta}^{\gamma}u_{\alpha\gamma}+\Gamma_{\bar{\alpha}\beta}^{\gamma}\Gamma_{\alpha\gamma}^{\rho}u_{\rho}+\Gamma_{\bar{\alpha}\beta}^{\gamma}\Gamma_{\alpha\gamma}^{\bar{\rho}}u_{\bar{\rho}},u_{\beta}\big)\\
\notag &\ \ \
+2ni\sum\limits_{\beta}\big(T(u_{\beta}),u_{\beta}\big)
+\sum\limits_{\alpha,\beta,\gamma,\rho}\bigg(-\Gamma_{\bar{\alpha}\alpha}^{\gamma}(u_{\gamma\beta}
+\Gamma_{\gamma\beta}^{\rho}u_{\rho}+\Gamma_{\gamma\beta}^{\bar{\rho}}u_{\bar{\rho}})
+\Gamma_{\alpha\bar{\alpha}}^{\bar{\gamma}}(u_{\bar{\gamma}\beta}+\Gamma_{\bar{\gamma}\beta}^{\rho}u_{\rho})
+W_{\bar{\alpha}}(\Gamma_{\alpha\beta}^{\gamma})u_{\gamma}\\
\notag &\ \ \ +W_{\bar{\alpha}}(\Gamma_{\alpha\beta}^{\bar{\gamma}})u_{\bar{\gamma}}
+\Gamma_{\alpha\beta}^{\gamma}(u_{\bar{\alpha}\gamma}+\Gamma_{\bar{\alpha}\gamma}^{\rho}u_{\rho})
+\Gamma_{\alpha\beta}^{\bar{\gamma}}(u_{\bar{\alpha}\bar{\gamma}}+\Gamma_{\bar{\alpha}\bar{\gamma}}^{\rho}u_{\rho}
+\Gamma_{\bar{\alpha}\bar{\gamma}}^{\bar{\rho}}u_{\bar{\rho}})+\Gamma_{\bar{\gamma}\gamma}^{\alpha}u_{\alpha\beta}
-\Gamma_{\bar{\alpha}\bar{\gamma}}^{\bar{\beta}}u_{\alpha\gamma},u_{\beta}\bigg)\\
\notag &\ \ \ -\sum\limits_{\alpha,\beta,\gamma}\big(\Gamma_{\bar{\alpha}\bar{\gamma}}^{\beta}u_{\alpha\gamma},u_{\bar{\beta}}\big)\\ &=-\sum\limits_{\alpha,\beta}\big(W_{\alpha}(u_{\bar{\alpha}\beta}),u_{\beta}\big)+2ni\sum\limits_{\beta}\big(T(u_{\beta}),u_{\beta}\big)
+\Sigma_1+\Sigma_2+\Sigma_3,
\end{align}
by using \eqref{SD} repeatedly, where $\Sigma_1$, $\Sigma_2$ and $\Sigma_3$ denote the summation of terms of type $({\ast}u_{\rho},u_{\beta})$, $({\ast}u_{\bar{\rho}},u_{\beta})$ and $({\ast}u_{ab},u_c)$ respectively. We see that
\begin{align}\label{eq5.5}
\notag \Sigma_1&=\sum\limits_{\alpha,\beta,\gamma,\rho}\Bigg(\big(W_{\bar{\alpha}}\Gamma_{\alpha\beta}^{\rho}-W_{\alpha}\Gamma_{\bar{\alpha}\beta}^{\rho}
-\Gamma_{\bar{\alpha}\beta}^{\gamma}\Gamma_{\alpha\gamma}^{\rho}+\Gamma_{\alpha\beta}^{\gamma}\Gamma_{\bar{\alpha}\gamma}^{\rho}
-\Gamma_{\bar{\alpha}\alpha}^{\gamma}\Gamma_{\gamma\beta}^{\rho}+\Gamma_{\alpha\bar{\alpha}}^{\bar{\gamma}}\Gamma_{\bar{\gamma}\beta}^{\rho}
+\Gamma_{\alpha\beta}^{\bar{\gamma}}\Gamma_{\bar{\alpha}\bar{\gamma}}^{\rho}\big)u_{\rho},u_{\beta}\Bigg)\\
&=(R_{\beta\ \bar{\alpha}\alpha}^{\ \rho}u_{\rho},u_{\beta})-2ni\sum\limits_{\beta,\rho}\big(\Gamma_{0\beta}^{\rho}u_{\rho},u_{\beta}\big),
\end{align}
by
\begin{align}\label{eq5.5a}
R_{\beta\ \bar{\alpha}\alpha}^{\ \rho}&=W_{\bar{\alpha}}\Gamma_{\alpha\beta}^{\rho}-W_{\alpha}\Gamma_{\bar{\alpha}\beta}^{\rho}-\Gamma_{\bar{\alpha}\alpha}^{\gamma}\Gamma_{\gamma\beta}^{\rho}
+\Gamma_{\alpha\bar{\alpha}}^{\bar{\gamma}}\Gamma_{\bar{\gamma}\beta}^{\rho}-\Gamma_{\bar{\alpha}\beta}^{\gamma}\Gamma_{\alpha\gamma}^{\rho}
+\Gamma_{\alpha\beta}^{\gamma}\Gamma_{\bar{\alpha}\gamma}^{\rho}+\Gamma_{\alpha\beta}^{\bar{\gamma}}\Gamma_{\bar{\alpha}\bar{\gamma}}^{\rho}
+2ni\Gamma_{0\beta}^{\rho},
\end{align}
by \eqref{eq2.8a}, $\Gamma_{\alpha\bar{\beta}}^{\gamma}=0$ in \eqref{QBD} and $J_{\bar{\alpha}\alpha}=\sum\limits_{\alpha}h(W_{\bar{\alpha}},JW_{\alpha})=i\sum\limits_{\alpha}\delta_{\alpha\bar{\alpha}}=ni$;
and
\begin{align}
\Sigma_2&=-\frac{i}{2}\sum\limits_{\alpha,\beta,\gamma,\rho}\bigg((W_{\bar{\alpha}}Q_{\beta\alpha}^{\bar{\rho}}
-\Gamma_{\bar{\alpha}\beta}^{\gamma}Q_{\gamma\alpha}^{\bar{\rho}}
-\Gamma_{\bar{\alpha}\alpha}^{\gamma}Q_{\beta\gamma}^{\bar{\rho}}
+\Gamma_{\bar{\alpha}\bar{\gamma}}^{\bar{\rho}}Q_{\beta\alpha}^{\bar{\gamma}})u_{\bar{\rho}},u_{\beta}\bigg)
=-\frac{i}{2}\sum\limits_{\alpha,\beta,\rho}(Q_{\beta\alpha,\bar{\alpha}}^{\bar{\rho}}u_{\bar{\rho}},u_{\beta})=0,
\end{align}
by
$\sum\limits_{\alpha,\beta,\rho}(Q_{\beta\alpha,\bar{\alpha}}^{\bar{\rho}}u_{\bar{\rho}},u_{\beta})=\int_MQ_{\beta\alpha,\bar{\alpha}}^{\bar{\rho}}u_{\bar{\rho}}u_{\bar{\beta}}dV
=-\int_MQ_{\rho\alpha,\bar{\alpha}}^{\bar{\beta}}u_{\bar{\rho}}u_{\bar{\beta}}dV=-\sum\limits_{\alpha,\beta,\rho}(Q_{\beta\alpha,\bar{\alpha}}^{\bar{\rho}}u_{\bar{\rho}},u_{\beta})$
by \eqref{eq2.9}; and
\begin{align}\label{eq5.5c}
\notag \Sigma_3&=\sum\limits_{\alpha,\beta,\gamma}\bigg(-\Gamma_{\bar{\alpha}\beta}^{\gamma}u_{\alpha\gamma}
-\Gamma_{\bar{\alpha}\alpha}^{\gamma}u_{\gamma\beta}+\Gamma_{\alpha\bar{\alpha}}^{\bar{\gamma}}u_{\bar{\gamma}\beta}
+\Gamma_{\alpha\beta}^{\gamma}u_{\bar{\alpha}\gamma}+\Gamma_{\alpha\beta}^{\bar{\gamma}}u_{\bar{\alpha}\bar{\gamma}}
+\Gamma_{\bar{\gamma}\gamma}^{\alpha}u_{\alpha\beta}-\Gamma_{\bar{\alpha}\bar{\gamma}}^{\bar{\beta}}u_{\alpha\gamma},u_{\beta}\bigg)\\
\notag &\ \ \ -\sum\limits_{\alpha,\beta,\gamma}\big(\Gamma_{\bar{\alpha}\bar{\gamma}}^{\beta}u_{\alpha\gamma},u_{\bar{\beta}}\big)\\
&=\sum\limits_{\alpha,\beta,\gamma}\big(\Gamma_{\alpha\bar{\alpha}}^{\bar{\gamma}}u_{\bar{\gamma}\beta}
+\Gamma_{\alpha\beta}^{\gamma}u_{\bar{\alpha}\gamma},u_{\beta}\big)
+\sum\limits_{\alpha,\beta,\gamma}\big(\Gamma_{\alpha\beta}^{\bar{\gamma}}u_{\bar{\alpha}\bar{\gamma}},u_{\beta}\big)
-\sum\limits_{\alpha,\beta,\gamma}\big(\Gamma_{\bar{\alpha}\bar{\gamma}}^{\beta}u_{\alpha\gamma},u_{\bar{\beta}}\big),
\end{align}
by using \eqref{eq2.9}. We also have
\begin{align}\label{eq5.5d}
\notag -\sum\limits_{\alpha,\beta}\big(W_{\alpha}(u_{\bar{\alpha}\beta}),u_{\beta}\big)
&=-\sum\limits_{\alpha,\beta}\big(u_{\bar{\alpha}\beta},W_{\alpha}^{\ast}u_{\beta}\big)
=\sum\limits_{\alpha,\beta}\bigg(u_{\bar{\alpha}\beta},W_{\bar{\alpha}}(u_{\beta})-\Gamma_{\bar{\gamma}\gamma}^{\alpha}u_{\beta}\bigg)\\
&=\sum\limits_{\alpha,\beta}\big(u_{\bar{\alpha}\beta},u_{\bar{\alpha}\beta}\big)
+\sum\limits_{\alpha,\beta,\gamma}\bigg(u_{\bar{\alpha}\beta},\Gamma_{\bar{\alpha}\beta}^{\gamma}u_{\gamma}
-\Gamma_{\bar{\gamma}\gamma}^{\alpha}u_{\beta}\bigg)
=\sum\limits_{\alpha,\beta}\big(u_{\bar{\alpha}\beta},u_{\bar{\alpha}\beta}\big)+\Sigma_4,
\end{align}
by Lemma \ref{lem3.1} and \eqref{SD}, where $\Sigma_4=\sum\limits_{\alpha,\beta,\gamma}(\Gamma_{\alpha\bar{\gamma}}^{\bar{\beta}}u_{\bar{\alpha}\gamma}
-\Gamma_{\gamma\bar{\gamma}}^{\bar{\alpha}}u_{\bar{\alpha}\beta},u_{\beta})$; and
\begin{equation}\label{eq5.5e}
\sum\limits_{\alpha}\big(T(u_{\alpha}),u_{\alpha}\big)-\sum\limits_{\alpha,\rho}\big(\Gamma_{0\alpha}^{\rho}u_{\rho},u_{\alpha}\big)
=\sum\limits_{\alpha}(u_{0\alpha},u_{\alpha}).
\end{equation}
holds by \eqref{SD}. Note that the first summation in the right hand side of \eqref{eq5.5c} is exactly $-\Sigma_4$ by \eqref{eq2.9}. Now substituting \eqref{eq5.5}-\eqref{eq5.5e} to \eqref{eq5.4}, we get $\int_Mu_{\alpha\beta}u_{\bar{\alpha}\bar{\beta}}dV$ equals to
\begin{align}
\notag \sum\limits_{\alpha,\beta}\big(u_{\bar{\alpha}\beta},u_{\bar{\alpha}\beta}\big)
+2ni\sum\limits_{\alpha}\big(u_{0\alpha},u_{\alpha}\big)-\sum\limits_{\alpha,\beta}\big(R_{\alpha\bar{\beta}}u_{\beta},u_{\alpha}\big) -\sum\limits_{\alpha,\beta,\rho}\big(\Gamma_{\alpha\gamma}^{\bar{\beta}}u_{\bar{\alpha}\bar{\gamma}},u_{\beta}\big)
-\sum\limits_{\alpha,\beta,\rho}\big(\Gamma_{\bar{\alpha}\bar{\gamma}}^{\beta}u_{\alpha\gamma},u_{\bar{\beta}}\big),
\end{align}
by using \eqref{eq2.9} for $\sum\limits_{\alpha,\beta,\gamma}\big(\Gamma_{\alpha\beta}^{\bar{\gamma}}u_{\bar{\alpha}\bar{\gamma}},u_{\beta}\big)
=-\sum\limits_{\alpha,\beta,\rho}\big(\Gamma_{\alpha\gamma}^{\bar{\beta}}u_{\bar{\alpha}\bar{\gamma}},u_{\beta}\big)$ and by \eqref{COMR} and \eqref{eq4.7} for
$R_{\beta\ \bar{\alpha}\alpha}^{\ \rho}=-R_{\beta\ \alpha\bar{\alpha}}^{\ \rho}=-R_{\beta\bar{\rho}\alpha\bar{\alpha}}=-R_{\alpha\bar{\rho}\beta\bar{\alpha}}=-R_{\alpha\ \beta\bar{\alpha}}^{\ \rho}=-R_{\beta\bar{\rho}}$. So we get
\begin{equation}
\notag i\int_Mu_{0\alpha}u_{\bar{\alpha}}dV=\frac{1}{2n}\int_M\bigg(u_{\alpha\beta}u_{\bar{\alpha}\bar{\beta}}-u_{\alpha\bar{\beta}}u_{\bar{\alpha}\beta}
+R_{\alpha\bar{\beta}}u_{\bar{\alpha}}u_{\beta}+\Gamma_{\alpha\gamma}^{\bar{\beta}}u_{\bar{\alpha}\bar{\gamma}}u_{\bar{\beta}}
+\Gamma_{\bar{\alpha}\bar{\gamma}}^{\beta}u_{\alpha\gamma}u_{\beta}\bigg)dV,
\end{equation}
The sum of this identity and its conjugation gives
\begin{align}\label{eq5.6}
\notag &\ \ \ i\int_Mu_{0\bar{\alpha}}u_{\alpha}dV-i\int_Mu_{0\alpha}u_{\bar{\alpha}}dV\\
&=\frac{1}{n}\int_M\bigg(u_{\alpha\bar{\beta}}u_{\bar{\alpha}\beta}-u_{\alpha\beta}u_{\bar{\alpha}\bar{\beta}}
-R_{\alpha\bar{\beta}}u_{\bar{\alpha}}u_{\beta}-\Gamma_{\alpha\gamma}^{\bar{\beta}}u_{\bar{\alpha}\bar{\gamma}}u_{\bar{\beta}}
-\Gamma_{\bar{\alpha}\bar{\gamma}}^{\beta}u_{\alpha\gamma}u_{\beta}\bigg)dV.
\end{align}
The result follows.

Comparing with the pseudohermitian case (cf. (9.37) in \cite{DT1}), the integral formula \eqref{eq5.6} has the extra term $\int_M\bigg(\Gamma_{\alpha\gamma}^{\bar{\beta}}u_{\bar{\alpha}\bar{\gamma}}u_{\bar{\beta}}+\Gamma_{\bar{\alpha}\bar{\gamma}}^{\beta}u_{\alpha\gamma}u_{\beta}\bigg)dV$, which is zero by $\Gamma_{\alpha\gamma}^{\bar{\beta}}=-\frac{i}{2}Q_{\gamma\alpha}^{\bar{\beta}}=0$ in the pseudohermitian case. Note that this extra term is the integral involving second-order covariant derivatives. We can transform them into the integral only involving first-order covariant derivatives via integration by parts in the following lemma.
\begin{lem}
\begin{equation}\label{eq5.7}
\int_M\Gamma_{\bar{\alpha}\bar{\gamma}}^{\beta}u_{\alpha\gamma}u_{\beta}dV
=\int_M\bigg(\frac{i}{2}Q_{\bar{\gamma}\bar{\beta},\alpha}^{\alpha}u_{\gamma}u_{\beta}
+\frac{1}{4}Q_{\alpha\gamma}^{\bar{\rho}}Q_{\bar{\beta}\bar{\gamma}}^{\rho}u_{\bar{\alpha}}u_{\beta}
-\frac{1}{2}Q_{\alpha\gamma}^{\bar{\rho}}Q_{\bar{\beta}\bar{\rho}}^{\gamma}u_{\bar{\alpha}}u_{\beta}\bigg)dV.
\end{equation}
\end{lem}
\begin{proof}
By \eqref{SD}, \eqref{eq3.52} and \eqref{SDCF}, we have
\begin{align}
\notag \int_M\Gamma_{\bar{\alpha}\bar{\gamma}}^{\beta}&u_{\alpha\gamma}u_{\beta}dV=\frac{i}{2}\int_MQ_{\bar{\gamma}\bar{\alpha}}^{\beta}u_{\alpha\gamma}u_{\beta}dV=\frac{i}{2}\int_MQ_{\bar{\gamma}\bar{\alpha}}^{\beta}u_{\gamma\alpha}u_{\beta}dV\\
\notag &=\frac{i}{2}\int_M\bigg(W_{\gamma}u_{\alpha}-\Gamma_{\gamma\alpha}^{\rho}u_{\rho}+\frac{i}{2}Q_{\alpha\gamma}^{\bar{\rho}}u_{\bar{\rho}}\bigg)
Q_{\bar{\gamma}\bar{\alpha}}^{\beta}u_{\beta}dV\\
\notag &=\frac{i}{2}\int_M\bigg(u_{\alpha}(-W_{\gamma}+\Gamma_{\rho\bar{\rho}}^{\bar{\gamma}})(Q_{\bar{\gamma}\bar{\alpha}}^{\beta}u_{\beta})
+\mathscr{E}\bigg)dV\\
\notag &=\frac{i}{2}\int_M\bigg(-W_{\gamma}(Q_{\bar{\gamma}\bar{\alpha}}^{\beta})u_{\alpha}u_{\beta}
-Q_{\bar{\gamma}\bar{\alpha}}^{\beta}u_{\alpha}W_{\gamma}(u_{\beta})
+\Gamma_{\rho\bar{\rho}}^{\bar{\gamma}}Q_{\bar{\gamma}\bar{\alpha}}^{\beta}u_{\alpha}u_{\beta}+\mathscr{E}\bigg)dV\\
\notag &=\frac{i}{2}\int_M\bigg(-W_{\gamma}(Q_{\bar{\gamma}\bar{\alpha}}^{\beta})u_{\alpha}u_{\beta}
-Q_{\bar{\gamma}\bar{\alpha}}^{\beta}u_{\alpha}\bigg(u_{\gamma\beta}+\Gamma_{\gamma\beta}^{\rho}u_{\rho}
-\frac{i}{2}Q_{\beta\gamma}^{\bar{\rho}}u_{\bar{\rho}}\bigg)
+\Gamma_{\rho\bar{\rho}}^{\bar{\gamma}}Q_{\bar{\gamma}\bar{\alpha}}^{\beta}u_{\alpha}u_{\beta}+\mathscr{E}\bigg)dV,
\end{align}
where $\mathscr{E}=-\Gamma_{\gamma\alpha}^{\rho}Q_{\bar{\gamma}\bar{\alpha}}^{\beta}u_{\beta}u_{\rho}
+\frac{i}{2}Q_{\alpha\gamma}^{\bar{\rho}}Q_{\bar{\gamma}\bar{\alpha}}^{\beta}u_{\beta}u_{\bar{\rho}}
=\Gamma_{\gamma\bar{\rho}}^{\bar{\alpha}}Q_{\bar{\gamma}\bar{\alpha}}^{\beta}u_{\beta}u_{\rho}
+\frac{i}{2}Q_{\alpha\gamma}^{\bar{\rho}}Q_{\bar{\gamma}\bar{\alpha}}^{\beta}u_{\beta}u_{\bar{\rho}}$ by \eqref{eq2.9}.
By
$$2Q_{\bar{\gamma}\bar{\alpha}}^{\beta}u_{\gamma\beta}
=Q_{\bar{\gamma}\bar{\alpha}}^{\beta}u_{\gamma\beta}-Q_{\bar{\beta}\bar{\alpha}}^{\gamma}u_{\beta\gamma}=0,$$
by \eqref{eq2.9} and \eqref{SDCF}, it equals to
\begin{align}
\notag &\ \ \ \frac{i}{2}\int_M\bigg(-W_{\gamma}Q_{\bar{\gamma}\bar{\alpha}}^{\beta}
-\Gamma_{\gamma\rho}^{\beta}Q_{\bar{\gamma}\bar{\alpha}}^{\rho}
+\Gamma_{\gamma\bar{\gamma}}^{\bar{\rho}}Q_{\bar{\rho}\bar{\alpha}}^{\beta}
+\Gamma_{\gamma\bar{\alpha}}^{\bar{\rho}}Q_{\bar{\gamma}\bar{\rho}}^{\beta}\bigg)u_{\alpha}u_{\beta}dV -\frac{1}{4}\int_M\bigg(Q_{\rho\gamma}^{\bar{\beta}}Q_{\bar{\gamma}\bar{\alpha}}^{\rho}+Q_{\rho\gamma}^{\bar{\beta}}Q_{\bar{\gamma}\bar{\rho}}^{\alpha}\bigg)u_{\alpha}u_{\bar{\beta}}dV\\
\notag &=-\frac{i}{2}\int_MQ_{\bar{\gamma}\bar{\alpha},\gamma}^{\beta}u_{\alpha}u_{\beta}dV
-\frac{1}{4}\int_M\bigg(Q_{\rho\gamma}^{\bar{\beta}}Q_{\bar{\gamma}\bar{\alpha}}^{\rho}+Q_{\rho\gamma}^{\bar{\beta}}Q_{\bar{\gamma}\bar{\rho}}^{\alpha}\bigg)u_{\alpha}u_{\bar{\beta}}dV.
\end{align}
Then \eqref{eq5.7} follows from the last identity, $Q_{\bar{\gamma}\bar{\alpha}}^{\rho}=Q_{\bar{\gamma}\bar{\rho}}^{\alpha}-Q_{\bar{\rho}\bar{\gamma}}^{\alpha}$ and $Q_{\beta\alpha}^{\bar{\gamma}}=-Q_{\gamma\alpha}^{\bar{\beta}}$ in \eqref{eq2.9}.
\end{proof}
Substituting \eqref{eq5.7} and its conjugation to \eqref{eq5.6}, we get \eqref{eq5.1}.
\subsection{The proof of \eqref{eq5.2}}
Note that
\begin{align}\label{eq5.8}
\notag \int_M(\Delta_{b}u)^2dV&=\int_M\bigg(\sum\limits_{\alpha}u_{\alpha\bar{\alpha}}+u_{\bar{\alpha}\alpha}\bigg)^2dV
=\int_M\sum\limits_{\alpha,\beta}(u_{\alpha\bar{\alpha}}u_{\beta\bar{\beta}}+u_{\alpha\bar{\alpha}}u_{\bar{\beta}\beta}+u_{\bar{\alpha}\alpha}u_{\beta\bar{\beta}}+u_{\bar{\alpha}\alpha}u_{\bar{\beta}\beta})dV\\
\notag &=\int_M\bigg(2\bigg|\sum\limits_{\alpha}u_{\alpha\bar{\alpha}}\bigg|^2+\sum\limits_{\alpha,\beta}(u_{\bar{\alpha}\alpha}+2ih_{\alpha\bar{\alpha}}u_0)u_{\beta\bar{\beta}}+\sum\limits_{\alpha,\beta}(u_{\alpha\bar{\alpha}}-2ih_{\alpha\bar{\alpha}}u_0)u_{\bar{\beta}\beta}\bigg)dV\\ &=\int_M\bigg(4\bigg|\sum\limits_{\alpha}u_{\alpha\bar{\alpha}}\bigg|^2+2ni\sum\limits_{\alpha}(u_0u_{\alpha\bar{\alpha}}-u_0u_{\bar{\alpha}\alpha})\bigg)dV.
\end{align}
On the other hand, we have
\begin{align}
\notag \int_M\sum\limits_{\alpha}u_0&(u_{\alpha\bar{\alpha}}-u_{\bar{\alpha}\alpha})dV
=\sum\limits_{\alpha}(u_0,u_{\bar{\alpha}\alpha}-u_{\alpha\bar{\alpha}})
=\sum\limits_{\alpha,\gamma}(u_0,W_{\bar{\alpha}}(u_{\alpha})-\Gamma_{\bar{\alpha}\alpha}^{\gamma}u_{\gamma}
-W_{\alpha}(u_{\bar{\alpha}})+\Gamma_{\alpha\bar{\alpha}}^{\bar{\gamma}}u_{\bar{\gamma}})\\
\notag &=\sum\limits_{\alpha}\big(W_{\bar{\alpha}}^{\ast}u_0,u_{\alpha}\big)-\sum\limits_{\alpha,\gamma}\big(\Gamma_{\alpha\bar{\alpha}}^{\bar{\gamma}}u_0,u_{\gamma}\big)
-\sum\limits_{\alpha}\big(W_{\alpha}^{\ast}u_0,u_{\bar{\alpha}}\big)+\sum\limits_{\alpha,\gamma}\big(\Gamma_{\bar{\alpha}\alpha}^{\gamma}u_0,u_{\bar{\gamma}}\big)\\
\notag &=-\sum\limits_{\alpha}\big(W_{\alpha}(u_0),u_{\alpha}\big)
+\sum\limits_{\alpha,\gamma}\big(\Gamma_{\gamma\bar{\gamma}}^{\bar{\alpha}}u_0,u_{\alpha}\big)
-\big(\Gamma_{\alpha\bar{\alpha}}^{\bar{\gamma}}u_0,u_{\gamma}\big)\\
\notag &\ \ \ +\sum\limits_{\alpha}\big(W_{\bar{\alpha}}(u_0),u_{\bar{\alpha}}\big)
-\sum\limits_{\alpha,\gamma}\big(\Gamma_{\bar{\gamma}\gamma}^{\alpha}u_0,u_{\bar{\alpha}}\big)
+\big(\Gamma_{\bar{\alpha}\alpha}^{\gamma}u_0,u_{\bar{\gamma}}\big)\\
\notag &=-\sum\limits_{\alpha}\big(W_{\alpha}(u_0),u_{\alpha}\big)+\sum\limits_{\alpha}\big(W_{\bar{\alpha}}(u_0),u_{\bar{\alpha}}\big)
=-\sum\limits_{\alpha}\big(u_{\alpha0},u_{\alpha}\big)+\sum\limits_{\alpha}\big(u_{\bar{\alpha}0},u_{\bar{\alpha}}\big)\\
\notag &=-\sum\limits_{\alpha,\beta}\big(u_{0\alpha}+A_{\alpha\beta}u_{\bar{\beta}},u_{\alpha}\big)+\sum\limits_{\alpha,\beta}\big(u_{0\bar{\alpha}}+A_{\bar{\alpha}\bar{\beta}}u_{\beta},u_{\bar{\alpha}}\big)\\
\notag &=\int_M\sum\limits_{\alpha}\big(u_{\alpha}u_{0\bar{\alpha}}-u_{\bar{\alpha}}u_{0\alpha}\big)dV
+\int_M\sum\limits_{\alpha,\beta}\big(A_{\bar{\alpha}\bar{\beta}}u_{\alpha}u_{\beta}-A_{\alpha\beta}u_{\bar{\alpha}}u_{\bar{\beta}}\big)dV.
\end{align}
Substitute this identity into \eqref{eq5.8} to get
\begin{align}
\notag \int_M\bigg((\Delta_{b}u)^2-4|\sum\limits_{\alpha}u_{\alpha\bar{\alpha}}|^2\bigg)dV=&2ni\sum\limits_{\alpha,\beta}\int_M\bigg((u_{\alpha}u_{0\bar{\alpha}}-u_{\bar{\alpha}}u_{0\alpha})
+(A_{\bar{\alpha}\bar{\beta}}u_{\alpha}u_{\beta}-A_{\alpha\beta}u_{\bar{\alpha}}u_{\bar{\beta}})\bigg)dV,
\end{align}
which is equivalent to \eqref{eq5.2}.

\section{The Proof of Theorem \ref{thm1.2}}
\begin{lem}\label{lem6.1}
For any $u\in{C_0^{\infty}(M)}$, we have
\begin{equation}
\notag \int_M\bigg(u_{\alpha}(\Delta_{b}u)_{\bar{\alpha}}+u_{\bar{\alpha}}(\Delta_{b}u)_{\alpha}\bigg)dV=-\int_M(\Delta_{b}u)^2dV.
\end{equation}
\end{lem}
\begin{proof}
By \eqref{eq6.1a}, \eqref{eq3.5a} and Lemma \ref{lem3.1}, we get
\begin{align}
\notag \int_Mu_{\bar{\alpha}}(\Delta_{b}u)_{\alpha}dV&=\int_MW_{\alpha}(\Delta_{b}u)u_{\bar{\alpha}}dV=(W_{\alpha}(\Delta_{b}u),u_{\alpha})
=\Bigg(\Delta_{b}u,\bigg(-W_{\bar{\alpha}}+\Gamma_{\bar{\beta}\beta}^{\alpha}\bigg)u_{\alpha}\Bigg)\\
\notag &=(\Delta_{b}u,-W_{\bar{\alpha}}(u_{\alpha})+\Gamma_{\bar{\alpha}\alpha}^{\beta}u_{\beta})=-(\Delta_{b}u,u_{\bar{\alpha}\alpha}),
\end{align}
by using \eqref{SD}. The summation of this identity and its conjugation gives the result.
\end{proof}
Note that
$$\int_M\big(\Delta_{b}f\big)dV=0,$$
holds for any $f\in{C_0^{\infty}}(M)$. Apply $f=\|\partial_{b}u\|^2\in{C_0^{\infty}}(M)$ to this identity for $u\in{C_0^{\infty}}(M)$. By the Bochner-type formula \eqref{eq1.1} and Lemma \ref{lem6.1}, we get
\begin{align}\label{eq6.1}
\notag 0&=\int_M\bigg(u_{\alpha\beta}u_{\bar{\alpha}\bar{\beta}}+u_{\alpha\bar{\beta}}u_{\bar{\alpha}\beta}+2i(u_{\alpha}u_{0\bar{\alpha}}-u_{\bar{\alpha}}u_{0\alpha})
+ni(A_{\bar{\alpha}\bar{\beta}}u_{\alpha}u_{\beta}-A_{\alpha\beta}u_{\bar{\alpha}}u_{\bar{\beta}})+R_{\alpha\bar{\beta}}u_{\bar{\alpha}}u_{\beta}\\ &\ \ \ -\frac{1}{2}(\Delta_{b}u)^2+\frac{i}{2}Q_{\bar{\alpha}\bar{\beta},\gamma}^{\gamma}u_{\alpha}u_{\beta}
-\frac{i}{2}Q_{\alpha\beta,\bar{\gamma}}^{\bar{\gamma}}u_{\bar{\alpha}}u_{\bar{\beta}}
-\frac{1}{4}Q_{\alpha\gamma}^{\bar{\rho}}Q_{\bar{\beta}\bar{\gamma}}^{\rho}u_{\bar{\alpha}}u_{\beta}\bigg)dV.
\end{align}

Apply \eqref{eq5.1} and \eqref{eq5.2} to get
\begin{align}\label{eq6.2}
\notag &\ \ \ i\int_M(u_{\alpha}u_{0\bar{\alpha}}-u_{\bar{\alpha}}u_{0\alpha})dV=Ci\int_M(u_{\alpha}u_{0\bar{\alpha}}-u_{\bar{\alpha}}u_{0\alpha})dV+(1-C)i\int_M(u_{\alpha}u_{0\bar{\alpha}}-u_{\bar{\alpha}}u_{0\alpha})dV\\
\notag &=\int_M\bigg(\frac{C}{n}u_{\alpha\bar{\beta}}u_{\bar{\alpha}\beta}-\frac{C}{n}u_{\alpha\beta}u_{\bar{\alpha}\bar{\beta}}-\frac{C}{n}R_{\alpha\bar{\beta}}u_{\bar{\alpha}}u_{\beta}\\
\notag &\ \ \ -\frac{C}{2n}iQ_{\bar{\alpha}\bar{\beta},\gamma}^{\gamma}u_{\alpha}u_{\beta}+\frac{C}{2n}iQ_{\alpha\beta,\bar{\gamma}}^{\bar{\gamma}}u_{\bar{\alpha}}u_{\bar{\beta}}
-\frac{C}{2n}Q_{\alpha\gamma}^{\bar{\rho}}Q_{\bar{\beta}\bar{\gamma}}^{\rho}u_{\bar{\alpha}}u_{\beta}
+\frac{C}{n}Q_{\alpha\gamma}^{\bar{\rho}}Q_{\bar{\beta}\bar{\rho}}^{\gamma}u_{\bar{\alpha}}u_{\beta}\\
&\ \ \ -\frac{2(1-C)}{n}\bigg|\sum\limits_{\alpha}u_{\alpha\bar{\alpha}}\bigg|^2+\frac{1-C}{2n}\big(\Delta_{b}u\big)^2
+(1-C)iA_{\alpha\beta}u_{\bar{\alpha}}u_{\bar{\beta}}-(1-C)iA_{\bar{\alpha}\bar{\beta}}u_{\alpha}u_{\beta}\bigg)dV.
\end{align}
Substituting \eqref{eq6.2} to \eqref{eq6.1}, we get
\begin{align}\label{eq6.3}
\notag 0&=\int_M\Bigg\{\bigg(1+\frac{2C}{n}\bigg)u_{\alpha\bar{\beta}}u_{\bar{\alpha}\beta}+\bigg(1-\frac{2C}{n}\bigg)u_{\alpha\beta}u_{\bar{\alpha}\bar{\beta}}
-\frac{4(1-C)}{n}\bigg|\sum\limits_{\alpha}u_{\alpha\bar{\alpha}}\bigg|^2\\
\notag &\ \ \ +\bigg(-\frac{1}{2}+\frac{1-C}{n}\bigg)(\Delta_{b}u)^2+\bigg(1-\frac{2C}{n}\bigg)R_{\alpha\bar{\beta}}u_{\bar{\alpha}}u_{\beta}
+i\bigg(n-2(1-C)\bigg)\bigg(A_{\bar{\alpha}{\bar{\beta}}}u_{\alpha}u_{\beta}-A_{\alpha\beta}u_{\bar{\alpha}}u_{\bar{\beta}}\bigg)\\
&\ \ \ +i\bigg(\frac{1}{2}-\frac{C}{n}\bigg)\bigg(Q_{\bar{\alpha}\bar{\beta},\gamma}^{\gamma}u_{\alpha}u_{\beta}-Q_{\alpha\beta,\bar{\gamma}}^{\bar{\gamma}}u_{\bar{\alpha}}u_{\bar{\beta}}\bigg)
-\bigg(\frac{1}{4}+\frac{C}{n}\bigg)Q_{\alpha\gamma}^{\bar{\rho}}Q_{\bar{\beta}\bar{\gamma}}^{\rho}u_{\bar{\alpha}}u_{\beta}
+\frac{2C}{n}Q_{\alpha\gamma}^{\bar{\rho}}Q_{\bar{\beta}\bar{\rho}}^{\gamma}u_{\bar{\alpha}}u_{\beta}\Bigg\}dV.
\end{align}
Let $\bigg(1+\frac{2C}{n}\bigg)\frac{1}{n}-\frac{4(1-C)}{n}=0$, i.e., $C=\frac{3n}{4n+2}$ in \eqref{eq6.3}.
By using $u_{\alpha\bar{\beta}}u_{\bar{\alpha}\beta}=\sum\limits_{\alpha,\beta}|u_{\alpha\bar{\beta}}|^2\ge\frac{1}{n}|\sum\limits_{\alpha}u_{\alpha\bar{\alpha}}|^2$ (cf. p. 489 in \cite{LW1}),
we get
\begin{align}\label{eq6.4}
\notag 0&\ge\int_M\frac{2(n-1)}{2n+1}u_{\alpha\beta}u_{\bar{\alpha}\bar{\beta}}-\frac{n^2-1}{n(2n+1)}(\Delta_{b}u)^2
+\frac{2(n-1)}{2n+1}R_{\alpha\bar{\beta}}u_{\bar{\alpha}}u_{\beta}
+i\frac{2n^2-2}{2n+1}\bigg(A_{\bar{\alpha}{\bar{\beta}}}u_{\alpha}u_{\beta}-A_{\alpha\beta}u_{\bar{\alpha}}u_{\bar{\beta}}\bigg)\\
\notag &\ \ \ +\frac{n-1}{2n+1}i\bigg(Q_{\bar{\alpha}\bar{\beta},\gamma}^{\gamma}u_{\alpha}u_{\beta}-Q_{\alpha\beta,\bar{\gamma}}^{\bar{\gamma}}u_{\bar{\alpha}}u_{\bar{\beta}}\bigg)
-\frac{2n+7}{8n+4}Q_{\alpha\gamma}^{\bar{\rho}}Q_{\bar{\beta}\bar{\gamma}}^{\rho}u_{\bar{\alpha}}u_{\beta}
+\frac{3}{2n+1}Q_{\alpha\gamma}^{\bar{\rho}}Q_{\bar{\beta}\bar{\rho}}^{\gamma}u_{\bar{\alpha}}u_{\beta}\bigg\}dV\\
\notag &=\int_M\bigg\{\frac{2(n-1)}{2n+1}u_{\alpha\beta}u_{\bar{\alpha}\bar{\beta}}+\frac{n^2-1}{n(2n+1)}\lambda_1u(\Delta_{b}u)\\
\notag &\ \ \ +\frac{2(n-1)}{2n+1}\bigg(R_{\alpha\bar{\beta}}u_{\bar{\alpha}}u_{\beta}
+i(n+1)\big(A_{\bar{\alpha}{\bar{\beta}}}u_{\alpha}u_{\beta}-A_{\alpha\beta}u_{\bar{\alpha}}u_{\bar{\beta}}\big)\\
&\ \ \ +\frac{i}{2}\big(Q_{\bar{\alpha}\bar{\beta},\gamma}^{\gamma}u_{\alpha}u_{\beta}-Q_{\alpha\beta,\bar{\gamma}}^{\bar{\gamma}}u_{\bar{\alpha}}u_{\bar{\beta}}\big)
-\frac{2n+7}{8(n-1)}Q_{\alpha\gamma}^{\bar{\rho}}Q_{\bar{\beta}\bar{\gamma}}^{\rho}u_{\bar{\alpha}}u_{\beta}
+\frac{3}{2(n-1)}Q_{\alpha\gamma}^{\bar{\rho}}Q_{\bar{\beta}\bar{\rho}}^{\gamma}u_{\bar{\alpha}}u_{\beta}\bigg)\bigg\}dV,
\end{align}
for $\Delta_{b}u=-\lambda_1u$. We use the following lemma to handle the second term in the right hand side.
\begin{lem}\label{lem6.2}
\begin{equation}
\notag -2\int_M\sum\limits_{\alpha}|u_{\alpha}|^2dV=\int_Mu(\Delta_{b}u)dV.
\end{equation}
\end{lem}
\begin{proof}
By Lemma \ref{lem3.1} and \eqref{SD}, we have
\begin{align}
\notag \int_M\sum\limits_{\alpha}|u_{\alpha}|^2dV&=\sum\limits_{\alpha}(u_{\alpha},u_{\alpha})
=\sum\limits_{\alpha}(W_{\alpha}u,u_{\alpha})=\sum\limits_{\alpha}(u,W_{\alpha}^{\ast}u_{\alpha})\\
\notag &=-\sum\limits_{\alpha}(u,W_{\bar{\alpha}}(u_{\alpha}))+\sum\limits_{\alpha,\beta}(u,\Gamma_{\bar{\beta}\beta}^{\alpha}u_{\alpha})\\
\notag &=-\sum\limits_{\alpha}(u,u_{\bar{\alpha}\alpha})-\sum\limits_{\alpha,\beta}(u,\Gamma_{\bar{\alpha}\alpha}^{\beta}u_{\beta})
+(u,\Gamma_{\bar{\beta}\beta}^{\alpha}u_{\alpha})=-\sum\limits_{\alpha}(u,u_{\bar{\alpha}\alpha}).
\end{align}
The summation of this identity and its conjugation gives the result.
\end{proof}
Let $X=u_{\bar{\alpha}}W_{\alpha}\in{T^{(1,0)}M}$. If \eqref{eq1.2} holds, by the definition of Ricci tensor, and $Tor$ given by \eqref{tor1} and $Q_1$, $Q_2$, $Q_3$ given by \eqref{GQ}, we get
\begin{equation}\label{eq6.5}
\begin{aligned}
&R_{\alpha\bar{\beta}}u_{\bar{\alpha}}u_{\beta}
+i(n+1)\big(A_{\bar{\alpha}{\bar{\beta}}}u_{\alpha}u_{\beta}-A_{\alpha\beta}u_{\bar{\alpha}}u_{\bar{\beta}}\big)\\
&+\frac{i}{2}\big(Q_{\bar{\alpha}\bar{\beta},\gamma}^{\gamma}u_{\alpha}u_{\beta}-Q_{\alpha\beta,\bar{\gamma}}^{\bar{\gamma}}u_{\bar{\alpha}}u_{\bar{\beta}}\big)
-\frac{2n+7}{8(n-1)}Q_{\alpha\gamma}^{\bar{\rho}}Q_{\bar{\beta}\bar{\gamma}}^{\rho}u_{\bar{\alpha}}u_{\beta}
+\frac{3}{2(n-1)}Q_{\alpha\gamma}^{\bar{\rho}}Q_{\bar{\beta}\bar{\rho}}^{\gamma}u_{\bar{\alpha}}u_{\beta}
{\ge}\kappa\sum\limits_{\alpha}|u_{\alpha}|^2.
\end{aligned}
\end{equation}
So by Lemma \ref{lem6.2}, \eqref{eq6.4} and \eqref{eq6.5}, we get
\begin{align}
0\ge\int_M\bigg(-\frac{2(n^2-1)}{n(2n+1)}\lambda_1+\frac{2(n-1)}{2n+1}\kappa\bigg)\sum\limits_{\alpha}|u_{\alpha}|^2dV,
\end{align}
i.e. $\lambda_1\ge\frac{\kappa{n}}{n+1}$. Theorem \ref{thm1.2b} is proved.

\bibliographystyle{plain}

\begin{thebibliography}{99}
 \bibitem{Ba1} {\sc E. Barletta}, {The Lichnerowicz theorem on CR manifolds}, {\it Tsukuba J. Math.}, \textbf{31}, (2007) no.1, 77-97.
 \bibitem{BaD1} {\sc E. Barletta and S. Dragomir}, {Differential equations on contact Riemannian manifolds}, {\it Ann. Sc. Norm. Super. Pisa}, Cl. Sci., Ser. IV, XXX (1) (2001), 63-96.
 \bibitem{BI1} {\sc O. Biquard}, {\it M\'etriques d'Einstein asymptotiquement sym\'etriques}, {Ast\'erisque} \textbf{265}, (2000).
 \bibitem{BD1} {\sc D.E. Blair and S. Dragomir}, {Pseudohermitian Geometry on Contact Riemann Manifolds}, {\it Rendiconti di Matematica}, Sereie VII Volume \textbf{22}, Roma (2002), 275-341.
 \bibitem{CC3} {\sc S. C. Chang and H. L. Chiu}, {On the estimate of the first eigenvalue of a sub-Laplacian on a pseudohermitian $3$-manifold}, Pacific J. Math, \textbf{232} (2007), no. 2, 269-282.
 \bibitem{CC1} {\sc S. C. Chang and H. L. Chiu}, {Nonnegativity of the CR Paneitz operator and its application to the CR Obata's theorem}, {\it J. Geom. Anal.} \textbf{19} (2009), 261-287.
 \bibitem{CC2} {\sc S. C. Chang and H. L. Chiu}, {On the CR analogue of Obata's theorem in a pseudohermitian $3$-manifold}, {\it Math. Ann.} \textbf{345} (2009), no.1, 33-51.
 \bibitem{Ch1} {\sc H. L. Chiu}, {The sharp lower bound for the first positive eigenvalue of the sub-Laplacian on a pseudohermitian 3-manifold}, {\it Ann. Global Anal. Geom.} \textbf{30} (2006), no. 1, 81-96.
 \bibitem{DP1} {\sc S. Dragomir and R. Petit}, {Contact harmonic maps}, {\it Differential Geom. Appl.} \textbf{30} (2012), no. 1, 65-84.
 \bibitem{DT1} {\sc S. Dragomir and G. Tomassini}, {\it Differential Geometry and Analysis on CR manifold}, {Progress in Math.} \textbf{246}, Birkh\"auser Boston, Inc., Boston, MA, (2006).
 \bibitem{FS1} {\sc G. B. Folland and E. M. Stein}, {Estimates for the $\bar{\partial}_b$-complex and analysis on the Heisenberg group}, {\it Comm. Pure Appl. Math.} \textbf{27} (1974), 429-522.
 \bibitem{GL1} {\sc C. R. Graham and J. M. Lee}, {Smooth solution of degenerate Laplacians on strictly pseudoconvex domains}, {\it Duke Math. Jour.} \textbf{57} (1988), no. 3, 697-720.
 \bibitem{G1} {\sc A. Greenleaf}, {The first eigenvalue of a sub-Laplacian on a pseudohermitian manifold},
     {\it Comm. Partial Differential Equations}, \textbf{10} (1985), 191-217.
 \bibitem{IPV1} {\sc S. Ivanov and A. Petkov and D. Vassilev}, {The sharp lower bound of the first eigenvalue of the sub-Laplacian on a quaternionic contact manifold}, {\it The Journal of Geometric Analysis} \textbf{24}, 2, (2014), 756-778.
 \bibitem{IV1} {\sc S. Ivanov and D. Vassilev}, {\it Extremals for the Sobolev inequality and the Quaternionic Contact Yamabe Problem}, Imprial College Press Lecture Notes, World Scientific Publishing Co. Pte. Ltd., Hackensack, NJ, (2011).
 \bibitem{IV2} {\sc S. Ivanov and D. Vassilev}, {An Obata type result for the first eigenvalue of the sub-Laplacian on a CR manifold with a divergence-free torsion}, {\it Journal of Geometry} \textbf{103}, 3, (2013), 475-504.
 \bibitem{JL3} {\sc J. M. Lee}, {The Fefferman metric and pseudohermitian invariants}, {\it Trans. Amer. Math. Soc.} \textbf{296} (1986) 411-429.
 \bibitem{LL1} {\sc S. Y. Li and H. S. Luk}, {The sharp lower bound for the first positive eigenvalue of a sub-Laplacian on a pseudo-Hermitian manifold}, {\it Proc. Amer. Math. Soc.} \textbf{132} (2004), no. 3, 789-798.
 \bibitem{LW1} {\sc S. Y. Li and X. D. Wang}, {An Obata-type theorem in CR geometry}, {\it J. Differential Geometry} \textbf{95} (2013), no.3, 483-502.
 \bibitem{Li1} {\sc A. Lichnerowicz}, {Geometrie des groupes de transformations}, {\it Travaux et Recherches Mathematiques, III.} {Dunod, Paris}, 1958.
 \bibitem{MS1} {\sc A. Menikoff and J. Sj\"ostrand}, {On the eigenvalues of a class of hypoelliptic operators}, {\it Math. Ann.}, \textbf{235} (1978), 55-58.
 \bibitem{Ob1} {\sc M. Obata}, {Certain conditions for a Riemannian manifold to be isometric with a sphere}, {\it J. Math. Soc. Japan} \textbf{14} (1962), no. 3, 333-340.
 \bibitem{Per1} {\sc D. Perrone}, {Contact metric manifolds whose characteristic vector field is a harmonic vector field}, {\it Differential Geom. Appl.} \textbf{20} (2004), no. 3, 367-378.
 \bibitem{Pe1} {\sc R. Petit}, {$Spin^c$-stuctures and Dirac operators on contact manifolds},  {\it Differential Geom. Appl.} \textbf{22} (2005), no. 2, 229-252.
 \bibitem{Se1} {\sc N. Seshadri}, {Approximately Einstein ACH metrics, volume renormalization, and an invariant for contact manifolds}, {\it Bull. Soc. Math. France} \textbf{137} (2009), no. 1, 63-91.
 \bibitem{TA1} {\sc N. Tanaka}, {A differential geometry study on strongly pseudo-convex manifolds}, {Kinokuniya Book Store Co., Ltd., Tokyo}, (1975).
  \bibitem{TA2} {\sc N. Tanaka}, {On a non-degenerate real hypersurfaces, graded Lie algebras and Cartan connections,} {\it Japan. J. Math.} \textbf{2} (1976) 131-190.
  \bibitem{T1} {\sc S. Tanno}, {Variational Problems on Contact Riemmannian Manifolds}, {\it Trans. Amer. Math. Soc.} \textbf{314} (1989), 349-379.
  \bibitem{Wa2} {\sc W. Wang}, {The Yamabe Problem on quaternionic contact manifolds}, {\it Ann. Mat. Pura Appl.}, \textbf{186} (2) (2007) 359-380.
 \bibitem{W1} {\sc S. M. Webster}, {Pseudohermitian structures on a real hypersurface}, {\it J. Differential Geometry} \textbf{13} (1978) 25-41.
\end{thebibliography}

\end{document}